\documentclass[10pt]{amsart}

\usepackage[latin1]{inputenc}
\usepackage{amsmath,amsfonts,amssymb,mathrsfs,amsthm}
\usepackage{mathrsfs}
\usepackage{xpatch}
\usepackage{color}
\newcommand{\C}{\mathscr{C}}
\newcommand{\D}{\mathscr{D}}
\newcommand{\HH}{\mathbb{H}}
\newcommand{\V}{\mathscr{V}}
\newcommand{\HHH}{\mathscr{H}}
\newcommand{\N}{\mathbb{N}}
\newcommand{\K}{\mathbb{K}}
\newcommand{\R}{\mathbb{R}}
\newcommand{\X}{\mathbb{X}}
\newcommand{\Y}{\mathscr{Y}}
\newcommand{\A}{\mathcal{A}}
\newcommand{\bs}{\boldsymbol}
\newcommand{\eps}{\varepsilon}
\newcommand{\Rot}{\nabla\!\times\!}
\newcommand{\lraup}{\relbar\joinrel\rightharpoonup}
\newtheorem{thm}{Theorem}[section]

\newtheorem{cond}{Assumption}[section]
\newtheorem{rem}{Remark}[section]
\newtheorem{lemma}{Lemma}[section]
\newtheorem{prop}{Proposition}[section]
\newtheorem{coro}{Corollary}[section]

\xpatchcmd{\proof}
{\itshape}
{\bfseries}
{}
{}

\begin{document}

\title[]{Evolutionary quasi-variational and variational inequalities with constraints on the derivatives} 
\author[F. Miranda]{Fernando Miranda}
\address{CMAT -- Departamento de Matemática, Escola de Ciências, Universidade do Minho, Campus de Gualtar, 4710-057 Braga, Portugal}
\email{fmiranda@math.uminho.pt}

\author[J.F. Rodrigues]{Jos\'e Francisco Rodrigues}
\address{CMAFcIO -- Departamento de Matemática, Faculdade de Ciências, Universidade de Lisboa
P-1749-016 Lisboa, Portugal}
\email{jfrodrigues@ciencias.ulisboa.pt}

\author[L. Santos]{Lisa Santos} 
\address{CMAT and Departamento de Matemática, Escola de Ciências, Universidade do Minho, Campus de Gualtar, 4710-057 Braga, Portugal}
\email{lisa@math.uminho.pt}

\begin{abstract}
This paper considers a general framework for the study of the existence of
quasi-variational and variational solutions to a class of nonlinear evolution
systems in convex sets of Banach spaces describing constraints on a
linear combination of partial derivatives of the solutions. The quasi-linear operators are
of monotone type, but are not required to be coercive for the existence of weak solutions, which
is obtained by a double penalisation/regularisation for the approximation of the solutions. In
the case of time-dependent convex sets that are independent of the solution, we show
also the uniqueness and the continuous dependence of the strong solutions of the variational
inequalities, extending previous results to a more general framework.
\end{abstract}

\maketitle

\section{Introduction}

While variational inequalities where introduced in 1964, by Fichera and Stampacchia in the framework of minimisation problems with obstacle constraints, the first evolutionary variational inequality was solved in the seminal paper of Lions and Stampacchia \cite{LiSt1967}, which was followed by many other works, including the extension to pseudo-monotone operators by Brézis in 1968 \cite{Bre1968} (see also \cite{Lions1969} or \cite{Roubicek2013}). 
Quasi-variational inequalities were introduced later by Bensoussan and Lions in 1973 to describe impulse control problems  \cite{BenLi1982} and were developed for several other mathematical models with free boundaries (see, for instance \cite{MignotPuel1977} and \cite{BaiCa1984}), mainly as implicit unilateral problems of obstacle type, in which the constraints  depend on the solution. 

The first physical models with gradient constraints formulated with quasi-variational inequalities of evolution type were proposed by Prighozhin, in \cite{Pri1996-1} and \cite{Pri1996-2}, respectively, for the sandpile growth and for the magnetization of type-II superconductors. 
This last model has motivated a first existence result for stationary problems in \cite{KunRod2000}, including other applications in elastoplasticity and in electrostatics, and, in \cite{RodriguesSantos2000}, in the parabolic framework for the $p$-Laplacian with an implicit gradient constraint, which was later extended to quasi-variational solutions for first order quasilinear equations in \cite{RodSan2012}, always in the scalar cases. 

In this work we consider weak solutions $\bs u=\bs u(x,t)$ to a class of quasi-variational inequalities associated with evolution equations or systems of the type 
\begin{equation}\label{L*}
\partial_t\bs u + L^* \bs a(L\bs u)+\bs b(\bs u)= \bs f
\end{equation}
formally in the unsaturated region of the scalar constraint 
\begin{equation*} 
|L\bs u|\le G[\bs u]
\end{equation*}
i.e., in the domain $\{(x,t):|L\bs u(x,t)| < G[\bs u](x,t)\}$, with a nonlocal positive and compact operator $G$, where $\partial_t\bs u$ denotes the partial time derivative, $L$ is a linear partial differential operator in $x$ with bounded coefficients and $L^*$ is its formal dual.
Here the monotone vector fields $\bs a$ and $\bs b$ are of power type growth, and the boundary value problems may be coercive or not. However in the region $\{(x,t):|L\bs u(x,t)|=G[\bs u](x,t)\}$ the equation~\eqref{L*}, in general, does not hold unless an extra term is added, raising interesting open questions. The general form of $L$ covers, in particular, the gradient, the Laplacian and higher order operators, the curl, the symmetric part of the Jacobian or classes of smooth vector fields, such as those of H\"ormander type. 
Weak quasi-variational solutions, which in general are non unique and do not have the time derivative in the dual space of the solution, are obtained by the passages to the limit of two vanishing parameters, one for an appropriate approximation/penalisation of the constraint on $L\bs u$ and a second one for a coercive regularisation, as in \cite{RodSan2012}. 
This method allows the application of the Schauder fixed point theorem to a general regularised two parameters variational equation of the type \eqref{L*} and extends considerably the work \cite{AzevedoSantos2010}.

When the constraint $G$, which may depend on time and space, is independent of the solution, i.e. $G=g(x,t)$, the problem becomes a variational one with the solution belonging to a time dependent convex set of a suitable Banach space. 
In this case, if the vector fields $\bs a$ and $\bs b$ are monotone, there exists uniqueness of the weak solution.
Under additional assumptions on the data, we show the existence, uniqueness and continuous dependence of the stronger solution of the corresponding evolution variational inequality, when the time derivative is actually a $L^2$ function. 
Here our method is adapted to gradient type constraints and it develops and extends the  pioneer work of \cite{San1991}, which was continued in \cite{San2002}, extended to a $p$-curl system in \cite{MirandaRodriguesSantos2012}, and to thick flows by \cite{Rodrigues2014} (see also \cite{MirandaRodrigues2016}). Although variational inequalities with time dependent convex sets have been studied in several works (see, for instance \cite{Kenmochi1981,Kubo2017} and their references), for the case of a convex with gradient constraint, only a few results have been stated, namely in \cite{Biroli1974}, as an application of abstract theorems, which assumptions are difficult to verify and, in general, require stronger hypothesis.

Recently, other approaches to evolutionary quasi-variational problems with gradient constraint have been developed by Kenmochi and co-workers in \cite{KanMurKen2009}, \cite{FuKen2013}, \cite{Ken2013} and  \cite{KenNie2016}, using variational evolution inclusions in Hilbert spaces with subdifferentials with a non-local dependence on parameters, and by Hinterm\"uller and Rautenberg in  \cite{HinRau2013}, using the pseudo-monotonicity and the $C_0$-semigroup approach of Brézis-Lions, and in \cite{HinRau2017}, using contractive iteration arguments that yield uniqueness results and numerical approximation schemes in interesting but special situations.
Although the elegant and abstract approach of \cite{HinRau2013} yields the existence of weak quasi-variational solutions under general stability conditions of Mosco type and a general scheme for the numerical approximation of a solution, the required assumptions for the existence theory are somehow more restrictive  than ours, in particular, in what concerns the required strong coercive condition. Other recent results on evolutionary quasi-variational inequalities can be found in \cite{Stefanelli2006} and \cite{Kubo2017}, both in more abstract frameworks and oriented to unilateral type problems and, therefore, with limited interest to constraints on the derivatives of the solutions. Recently, in \cite{HinRauStr2018}, a semidiscretization in time, with monotone non-decreasing data, was used to obtain non-decreasing in time solutions to quasi-variational inequalities with gradient constraints, including an interesting numerical scheme.

This paper is organised as follows: in Section 2 we state our framework and the main results on the existence of weak quasi-variational solutions and on the well-posedness of the strong variational solutions; in Section 3 we illustrate the nonlocal constraint operator $G$ and the linear partial differential operator $L$ with several examples of applications; Section 4 deals with the approximated problem and {\em a priori} estimates; the proof of the existence of the weak quasi-variational solutions is given in Section 5 and, finally in Section 6 we show the uniqueness and the continuous dependence on the data in the variational inequality case.

\section{Assumptions and main results}

Let $\Omega$ be a bounded open subset of $\R^d$, with a Lipschitz boundary, $d\geq2$ and for $t\in(0,T]$ we  denote $Q_t=\Omega\times(0,t)$. 
For a real vector function $\bs u=\bs u(x,t)=(u_1,\ldots,u_m)$, $(x,t)\in Q_T$, and a multi-index $\alpha=(\alpha_1,\ldots,\alpha_d)$, with $\alpha_1,\ldots,\alpha_d\in\N_0$ and $\alpha_1+\cdots+\alpha_d=|\alpha|$, we denote $\partial^\alpha u_i = \frac{\partial^{|\alpha|}u_i}{\partial{x_1^{\alpha_1}}\cdots\partial x_d^{\alpha_d}}$ the partial derivatives of $u_i$. 
Given real numbers $a,b$ we set $a\vee b=\max\{a,b\}$ and $a\wedge b=\min\{a,b\}$.

We introduce now several assumptions which will be important to set the functional framework of our problem.

\begin{cond}\label{op:L}
For $p\in[1,\infty]$, let $L$ be a linear differential operator of order $s\ge1$, given by	
\begin{equation*} 
L:\bs V_p\to L^p(\Omega)^\ell \text{ such that }
\bs V_p=\{\bs u\in L^p(\Omega)^m:L\bs u\in L^p(\Omega)^\ell \}
\end{equation*}
is endowed with the graph norm, $\ell, m\in\N$.
\end{cond}

In general, the operator $L$ can have the form 
\begin{equation*}
(L\bs u)_j=\sum_{|\alpha|\leq s}\sum_{k=1}^m\lambda_{\alpha,k}^j \partial^\alpha u_k,
\end{equation*}
for $j=1,\ldots, \ell,$ $\alpha=(\alpha_1,\dots,\alpha_d)\in \N_0^d$ is a multi-index and  each  $\lambda_{\alpha,k}^j\in L^\infty(\Omega)$ but we shall consider mainly the following four illustrative examples with constant coefficients, although we can consider also their generalisations with variable coefficients, as in the fifth example:
{\it \begin{enumerate}
\item[1.] $L u=\nabla u$ (gradient of $u$; $m=1$, $\ell=d$ );
\item[2.] $L u=\Delta u$ (Laplacian of $u$; $m=\ell=1$);
\item[3.] $L\bs u=\Rot\bs u$ (curl of $\bs u$; $d=m=\ell=3$);
\item[4.] $L\bs u=D\bs u = \frac12(\nabla\bs u + \nabla\bs u^T)$ (symmetric part of the Jacobian of $\bs u$; $d=m$, $\ell=m^2$);
\item[5.] $\bs{L} u=(X_1 u,\dots,X_\ell u)$, where $X_j=\displaystyle\sum_{i=1}^d \alpha_{ji}\tfrac{\partial\ }{\partial{x_i}}$, where  $\alpha_{ji}$ are appropriate scalar real functions (subelliptic gradient of $u$; $m=1$, $1\le j \le \ell$,  $1\le i \le d$).
\end{enumerate}}
	
\begin{cond}\label{op:a}  
Let $\bs a:Q_T\times\R^\ell\rightarrow\R^\ell$ and $\bs b:Q_T\times\R^m\rightarrow\R^m$ be Carath\'eodory functions, i.e.,  they are measurable functions in  the variables $(x,t)$, for all $\bs \xi\in\R^\ell$ and $\bs\eta\in\R^m$, respectively,  and they are continuous in the variables $\bs \xi\in\R^\ell$ and $\bs\eta\in\R^m$, for a.e. $(x,t)\in Q_T$.
Suppose, additionally, that $\bs a$ and $\bs b$ satisfy the following structural conditions: for all $\bs\xi,\bs\xi'\in\R^\ell$ and $\bs \eta,\bs\eta'\in\R^m$ and a.e. $(x,t)\in Q_T$, 
\begin{subequations} 
\begin{align}
\label{zero}
&|\bs a(x,t,\bs \xi)|\leq a^{*}|\bs \xi|^{p-1},\\
\label{op:a:prop:c}
&(\bs a(x,t,\bs\xi)-\bs a(x,t,\bs\xi'))\cdot(\bs\xi-\bs\xi')\ge0,\quad\quad &\\
\label{op:a:prop:b}
&|\bs b(x,t,\bs \eta)|\leq b^{*}|\bs \eta|^{(p-1)\vee 1},\\
\label{op:c:prop:c}
&(\bs b(x,t,\bs\eta)-\bs b(x,t,\bs\eta'))\cdot(\bs\eta-\bs\eta')\ge0,
\end{align}
where  $a^*$  and $b^*$ are positive constants and $1<p<\infty$.
\end{subequations}
\end{cond}

\begin{cond}\label{Xp}
For a given $p\in(1,\infty)$, we work with a closed subspace $\X_p$ of $\bs V_p$ such that $\X_p\subset L^2(\Omega)^m$ and  $\|\bs v\|_{\X_p}:=\|L\bs v\|_{L^p(\Omega)^\ell } $  is a norm in $\X_p$ equivalent to the norm induced from $\bs V_p$.  
\end{cond}
	
\begin{rem}
For simplicity, in this work we consider a functional framework where we suppose valid the Poincar\'e and Sobolev type inequalities, as in the Dirichlet problems of the five examples.
However our approach is still valid for more general frameworks to include Neumann and mixed type boundary conditions. 
\end{rem}
	
\begin{cond} \label{gelfand}
There exists a Hilbert subspace $\HH$ of $ L^2(\Omega)^m$  such that $\big(\X_p, \HH, \X'_p\big)$ is a Gelfand triple and the inclusion of $\X_p$ into $\HH$ is compact for the given $p$, $1<p<\infty$.
\end{cond}
	
From now on, we denote
\begin{equation*} 
\V_p=L^p\big(0,T;\X_p\big), \quad \HHH=L^p(0,T;\HH), \quad 1\leq p\leq\infty,
\end{equation*}
and we observe that  $L^{p'}\big(0,T; \X'_p\big)=\V'_p$, with $p'=\frac{p}{p-1}$ for $1<p<\infty$.

By well known embedding theorems on Sobolev-Bochner spaces (see Chapter 7 of \cite{Roubicek2013}, for instance), we have 
\begin{equation}\label{embed}
\Y_p=\big\{\bs v\in\V_p:\partial_t\bs v\in\V_p'\big\}\subset\C\big([0,T];\HH\big)
\end{equation}
and Assumption~\ref{gelfand} implies, by Aubin-Lions lemma, that the embedding $\Y_p\hookrightarrow\HHH$ is also compact for $1<p<\infty$.

\begin{cond}\label{Phi}
We consider  a nonlinear continuous functional $G:\HHH\rightarrow L^1(Q_T)$, such that its restriction to $\V_p$ is compact with values in $\C\big([0,T];L^\infty(\Omega)\big)$, i.e., $G:\V_p\rightarrow \C\big([0,T];L^\infty(\Omega)\big)$ is compact.
In addition, we assume  
\begin{equation*}
0<g_*\le G[\bs u](x,t)\le g^*\text{ for  all }\bs u\in\V_p, \text{ for all } t\in [0,T]\text{ and a.e. } x\in\Omega
\end{equation*}
for given constants $g_*$ and $g^*$. 
\end{cond}
Since $G$ is compact in $\V_p$, in particular, for any sequence $\{\bs v_n\}_n$ weakly convergent to $\bs v$ in $\V_p$, there exists a subsequence, still denoted by $\{\bs v_n\}_n$, such that $\{G[\bs v_n]\}_n$ converges uniformly to $G[\bs v]$ in $\C\big([0,T];L^\infty(\Omega)\big)$.

For $\bs v\in\V_p$ and a.e. $t\in(0,T)$ we define the nonempty convex set for $G[\bs v](t)\in L^\infty(\Omega)$
\begin{equation}\label{kapa}
\K_{G[\bs v](t)}=\big\{\bs w\in\X_p:|L\bs w|\le G[\bs v](t)\big\},
\end{equation}
where $|\cdot|$ is the euclidean norm in $\R^\ell $ and we denote $\bs w\in\K_{G[\bs v]}$ iff $\bs w(t)\in\K_{G[\bs v](t)}$ for a.e. $t\in(0,T)$.

For $1\le p< \infty$, we denote the duality pairing between $\X'_p$ and $\X_p$  by  $\langle\,\cdot\,,\,\cdot\,\rangle_p$ and we consider the quasi-variational inequality associated with \eqref{L*} and \eqref{kapa}. Find $\bs u\in \V_p$ satisfying
\begin{equation}\label{iqv}
\left\{\begin{array}{l}
\bs u\in\K_{G[\bs u]},\vspace{3mm}\\
\displaystyle\int_0^T\langle\partial_t\bs v,\bs v-\bs u\rangle_p+\int_{Q_T} \bs a(L\bs u)\cdot L(\bs v-\bs u)
+\int_{Q_T}\bs b(\bs u)\cdot(\bs v -\bs u)
\vspace{3mm}\\
\hfill{\ge\displaystyle\int_{Q_T}\bs f\cdot(\bs v-\bs u)-\frac12\int_\Omega|\bs v(0)-\bs u_0|^2,}\vspace{3mm}\\
\hfill{\forall\,\bs v\in \Y_p\text{ such that }
\bs v\in\K_{G[\bs u]}.}
\end{array}
\right.
\end{equation}

\begin{thm}\label{IQV0}
Suppose that Assumptions \ref{op:L} to \ref{Phi} are satisfied,  $\bs f\in L^2(Q_T)^m$ and $\bs u_0\in \K_{G[\bs u_0]}$. 
Then the quasi-variational inequality \eqref{iqv} has a weak solution
$\bs u\in\V_p\ \cap\  L^\infty\big(0,T;L^2(\Omega)^m\big)$.
\end{thm}

We note that, by Assumption~\ref{Phi} the solutions have bounded $L\bs u$ but, in general, this may not imply that $\bs u$ is itself bounded.
We also observe that, by insufficient regularity in time, we could not guarantee that the weak solution u satisfies the initial condition in the classical sense, but only in the generalised sense \eqref{iqv} as in \cite{Bre1968} and \cite{Lions1969}.

We consider a positive bounded function $g:Q_T\rightarrow\R^+$ and the special case of the convex set \eqref{kapa}, with 
$$G[\bs v](x,t)=g(x,t)\text{ for a.e. }t\in(0,T).$$
So, 
\begin{equation}\label{kapaIV}
\bs v\in\K_g \text{ iff }\bs v(t)\in\K_{g(t)}=\big\{\bs v\in \X^p:|L\bs v|\le g(t)\big\}\text{ for a.e. }t\in(0,T).
\end{equation}
In this case, the convex being independent of the solution, the problem becomes variational and the weak solution of Theorem \ref{IQV0} is unique by the following theorem.

\begin{thm} \label{th2.2}
The variational inequality \eqref{iqv}  with a fixed convex $\K_g$, as in \eqref{kapaIV}, for a given strictly positive function $g=g(x,t)\in\C\big([0,T];L^\infty(\Omega)\big)$, $\bs u_0\in\K_{g(0)} $ and $\bs f\in L^2(Q_T)^m$, has at most one solution provided $\X_p\subset\HH$ and one of the monotonicity conditions is strict, i.e.
\begin{equation*}
\big(\bs b(x,t,\bs\eta)- \bs b(x,t,\bs\eta')\big)\cdot(\bs\eta-\bs\eta')>0\text{ for }\bs\eta\not=\bs\eta',
\end{equation*}
or
\begin{equation*}
\big(\bs a(x,t,\bs\xi)- \bs a(x,t,\bs\xi')\big)\cdot(\bs\xi-\bs\xi')>0\text{ for }\bs\xi\not=\bs\xi',
\end{equation*}
and Assumption~\ref{Xp} holds.
\end{thm}

We can now introduce the strong formulation of the corresponding variational inequality.
Find $\bs w \in \V_p\cap H^1\big(0,T;L^2(\Omega)^m\big)$ satisfying, for all $t\in (0,T]$:
\begin{equation}\label{iv}
\left\{\begin{array}{l}
\bs w\in\K_{g},\ \bs w(0)=\bs w_0,\vspace{3mm}\\
\displaystyle\int_{Q_t}\partial_t\bs w\cdot(\bs v-\bs w)+\int_{Q_t} \bs a(L\bs w)\cdot L(\bs v-\bs w)+\int_{Q_t}\bs b(\bs w)\cdot(\bs v -\bs w)
\hspace{2cm}\vspace{3mm}\\
\hfill{\displaystyle\ge\displaystyle\int_{Q_t}\bs f\cdot(\bs v-\bs w)},\quad\forall\,\bs v\in \K_{g}\cap\V_p.
\end{array}
\right.
\end{equation}
We observe that if $\bs w$ is a strong solution to \eqref{iv} it is also a weak solution to \eqref{iqv}.
Indeed, if we take $\bs v\in\Y_p\cap\K_g\subset\C\big(0,T;L^2(\Omega)^m\big)$ in \eqref{iv} with $t=T$, since 
\begin{equation*}
\int_0^T\langle\partial_t\bs w-\partial_t\bs v,\bs v-\bs w\rangle_p=\tfrac12\int_\Omega|\bs v(0)-\bs w_0|^2-\tfrac12\int_\Omega|\bs v(T)-\bs w(T)|^2\leq \tfrac12\int_\Omega|\bs v(0)-\bs w_0|^2,
\end{equation*}
we immediately conclude that $\bs w$ also satisfies \eqref{iqv}.

We consider also a stronger noncoercive framework with a potential vector field  $\bs a$ and a lower order term $\bs b$ with linear growth, by replacing the Assumption~\ref{op:a} by the following.
\begin{cond}\label{cond:bis}
Let $\bs a:Q_T\times\R^\ell\rightarrow\R^\ell$ and $\bs b:Q_T\times\R^m\rightarrow\R^m$ be Carath\'eodory functions, i.e.,  they are measurable functions in  the variables $(x,t)$, for all $\bs \xi\in\R^\ell$ and $\bs\eta\in\R^m$, respectively and they are continuous in the variables $\bs \xi\in\R^\ell$ and $\bs\eta\in\R^m$, respectively, for a.e. $(x,t)\in Q_T$.
Suppose, additionally, that there exists $A:Q_T\times\R^\ell\rightarrow\R$, such that, for all $\bs\xi\in\R^\ell$ and a.e.\ $(x,t)\in\overline Q_T$, $A$ is differentiable in $t$ and in  $\bs\xi$, and 
\begin{subequations} 
\begin{equation}\label{apot}
A=A(x,t,\bs\xi)\text{ is convex in }\bs\xi,\quad \nabla_{\bs\xi}A =\bs a,
\end{equation}
\begin{align}
\label{op:a:bis}
&0\leq A(x,t,\bs\xi)\leq a^* |\bs\xi|^p,\quad |\partial_t A(x,t,\bs\xi)|\leq A_1+A_2|\bs\xi|^p,\\
\intertext{and $\bs b$ satisfies the following structural conditions: for all  $\bs \eta,\bs\eta'\in\R^m$ and a.e. $(x,t)\in Q_T$,}
%
&(\bs b(x,t,\bs\eta)-\bs b(x,t,\bs\eta'))\cdot(\bs\eta-\bs\eta')\ge0,\\
&|\bs b(x,t,\bs \eta)|\leq b^{*}|\bs \eta|,
\label{op:a:prop:b:bis}
\end{align}
where $a^*$, $A_1$, $A_2$ and  $b^*$ are positive constants, $1<p<\infty$.
\end{subequations}
\end{cond}

In the non-coercive case, we have the well-posedeness result on the existence, uniqueness and continuous dependence of the (strong) variational solution \eqref{iv}. Under the additional strong monotonicity assumption, for instance, for operators of $p$-Laplacian type, when $\bs a(\bs\xi)=|\bs \xi|^{p-2}\bs\xi$, the continuous dependence result in the coercive case also holds in the space $\V_p$.

\begin{thm}\label{IV0}
Suppose that Assumptions \ref{op:L}, \ref{cond:bis}, \ref{Xp} and  \ref{gelfand} are satisfied and
\begin{equation}\label{assumptions_iv}
\bs f\in L^2(Q_T)^m,\quad \bs w_0\in \K_{g(0)},\quad g\in W^{1,\infty}\big(0,T;L^\infty(\Omega)\big)\text{ with }g\ge g_*>0.
\end{equation}
Then the variational inequality \eqref{iv} has a unique solution $\bs w\in\V_p\cap H^1\big(0,T;L^2(\Omega)^m\big)$.
\end{thm}

\begin{thm}\label{stability} 
Suppose that the assumptions  of Theorem~\ref{IV0} hold and for $i=1,2$, let $\bs w_i$ be the solution to the variational inequality \eqref{iv} with data $\bs f_i,\bs w_{i_0},g_i$ satisfying \eqref{assumptions_iv}.
Then there exists a positive constant $C=C(T)$, such that,
\begin{equation} \label{ContDep1}
	\|\bs w_{1}-\bs w_{2}\|^2_{L^\infty(0,T;L^{2}(\Omega)^m)}\le
	C\Big(\|\bs f_1-\bs f_2\|^2_{L^{2}(Q_T)^m}
	+\|\bs w_{1_0}-\bs w_{2_0}\|^2_{L^2(\Omega)^m} 
	+\|g_1-g_2\|_{L^1(0,T;L^\infty(\Omega))}\Big).
\end{equation}

If , in addition, $\bs a$ satisfies
\begin{equation} \label{a:monot:forte}
(\bs a(x,t,\bs\xi)-\bs a(x,t,\bs\xi'))\cdot(\bs\xi-\bs\xi')\ge
\begin{cases}
a_*|\bs \xi-\bs \xi'|^p&\text{ if }p\ge 2,\\
a_*\big(|\bs \xi|+|\bs \xi'|\big)^{p-2}|\bs \xi-\bs \xi'|^2&\text{ if }p<2,
\end{cases}
\end{equation}
where $a_*$ is a positive constant depending on $p$, $1<p<\infty$,
then there exists $C_*=C(a_*,p,T)>0$ such that
\begin{multline} \label{ContDep22}
\|\bs w_{1}-\bs w_{2}\|^{2}_{L^\infty(0,T;L^{2}(\Omega)^m)}
+\|\bs w_1-\bs w_2\|^{2\vee p}_{\V^p}\\
\leq
C_*\Big(\|\bs f_1-\bs f_2\|^{2}_{L^{2}(Q_T)^m}
+\|\bs w_{1_0}-\bs w_{2_0}\|^2_{L^2(\Omega)^m} 
+\|g_1-g_2\|_{L^1(0,T;L^\infty(\Omega))}\Big).
\end{multline}
\end{thm}

\begin{rem}
For strong solutions $\bs w\in\K_g\cap H^1\big(0,T;L^2(\Omega)^m\big)$ the variational inequality \eqref{iv} for a.e. $t\in(0,T)$ is equivalent to
\begin{multline}\label{iv_espaco}
\int_\Omega\partial_t\bs w(t)\cdot(\bs z-\bs w(t))+\int_\Omega\bs a(t,L\bs w(t))\cdot L(\bs z-\bs w(t))\\
+\int_\Omega\bs b(t,\bs w(t))\cdot(\bs z-\bs w(t))\ge\int_\Omega\bs f(t)\cdot(\bs z-\bs w(t)),\quad\forall\bs z\in\K_{g(t)},
\end{multline}
provided we assume $g\in\C\big([0,T];L^\infty(\Omega)\big)$, $g\ge g_*>0$. Indeed, for arbitrary $\delta>0$, $0<\delta<t<T-\delta,$ for fixed $t\in(0,T)$, we set $\displaystyle\eps_\delta=\sup_{t-\delta<\tau< t+\delta}\|g(t)-g(\tau)\|_{L^\infty(\Omega)}$ and we may define
\begin{equation*}
\bs v(\tau)=\begin{cases}\bs0&\text{ if }\tau\not\in(t-\delta,t+\delta),\\
\frac{g_*}{g_*+\eps_\delta}\bs z&\text{ if }\tau\in(t-\delta,t+\delta),
\end{cases}
\end{equation*}
which is such that $\bs v\in\V_p\cap\K_g$ whenever $\bs z\in\K_{g(t)}$. Hence we can choose this $\bs v$ as test function in \eqref{iv} with $t=T$, divide by $2\delta$ and let $\delta\rightarrow0$ obtaining, by Lebesgue's theorem, the inequality \eqref{iv_espaco} for a.e. $t\in(0,T)$.
\end{rem}

As a Corollary of Theorems \ref{th2.2} and \ref{IV0}, we can drop the differentiability in time of $g$ and still obtain an existence and uniqueness result for the weak variational inequality \eqref{iqv} with $\K_g$, extending \cite[Theorem~3.8]{Ken2013}.

\begin{thm}\label{IVweak} Suppose that Assumptions \ref{op:L}, \ref{cond:bis}, \ref{Xp} and \ref{gelfand} are satisfied and
\begin{equation*}
\bs f\in L^2(Q_T)^m,\quad g\in\C\big([0,T];L^\infty(\Omega)\big)\text{ with } g\ge g_*>0,\quad\bs w_0\in\K_{g(0)}.
\end{equation*}
	
Then the variational inequality \eqref{iqv} for $\K_g$  has a unique weak solution $\bs w\in\V_p\cap\C\big([0,T];L^2(\Omega)^m\big)$.
\end{thm}

\section{Applications with particular $G$ and $L$}

In this section we present some examples of compact nonlocal operators $G$, satisfying the Assumption~\ref{Phi}, and linear operators $L$, satisfying the Assumption~\ref{op:L}.

\subsection{Nonlocal compact operators} $ $

Here we are interested in two examples of compact operators $G$ given in the form
\begin{equation} \label{G0}
G[\bs v]=g(x,t,\bs\zeta(\bs v)(x,t))\ \text{ a.e. in }Q_T,
\end{equation}
where $g=g(x,t,\bs\zeta):Q_T\times\R^m\rightarrow\R$ is a positive function, continuous in $(x,t)\in Q_T$ and in $\bs\zeta\in\R^m$, and $\bs\zeta:\V_p\rightarrow \C(\overline Q_T)^m$ is a completely continuous mapping.

\subsubsection{\bf {\footnotesize Regularization by integration in time}} $ $

We define the compact operator by
\begin{equation*}
\bs\zeta(\bs v)(x,t)=\int_0^t\bs v(x,s)K(t,s)ds,\quad\text{ for a.e. }(x,t)\in Q_T,
\end{equation*}
where $K=K(t,s)$ is a given kernel satisfying
\begin{equation} \label{G2}
K, \partial_t K\in L^\infty\big((0,T)\times(0,T)\big).
\end{equation}

For simplicity, we assume here the existence of a constant $g_*$ and a real bounded function $g^*$ such that
\begin{equation*} 
0<g_*\le g(x,t,\bs\xi)\le g^*(M)\quad\text{ for a.e. }(x,t)\in Q_T,\ \forall\bs\xi:|\bs\xi|\le M.
\end{equation*}

We also assume that the embedding
\begin{equation} \label{G4}
\X_p\hookrightarrow \C(\overline\Omega)^m\ \text{ is compact},
\end{equation}
which, by the Rellich-Kondratchov theorem, is satisfied if $\X_p\subset W^{s,q}(\Omega)^m$ with $q>\frac{d}{s}$.

Let $\V_p=L^p\big(0,T;\X_p\big)$, $p>1$, and observe that, by assumption \eqref{G2}, not only $\bs\zeta(\bs v)\in\V_p$ but also $\partial_t\bs\zeta(\bs v)\in\V_p$, i.e.
$$\bs\zeta(\bs v)\in W^{1,p}(0,T;\X_p).$$
Hence, by Lemma 2.2 of \cite{BarrettSuli2012}, for instance, the image by $\bs\zeta$ of a bounded subset of $\V_p$, being bounded in $W^{1,p}(0,T;\X_p)$, by \eqref{G4} is relatively compact in $\C(\overline Q_T)^m$. So $\bs\zeta:\V_p\rightarrow\C(\overline Q_T)^m$ is a completely continuous mapping and, therefore, $G$ defined in \eqref{G0} satisfies Assumption \ref{Phi}.

\subsubsection{\bf {\footnotesize Coupling with a nonlinear parabolic equation}}$ $

We may define the compact operator through the unique solution of the Cauchy-Dirichlet problem for the quasilinear parabolic scalar equation
\begin{align}
\label{p1}&\partial_t\zeta-\nabla\cdot \bs a(x,t,\nabla\zeta)=\varphi_{\bs v}\quad\text{in }Q_T,\\
\label{p2}&\zeta=0\quad\text{on }\partial\Omega\times(0,T),\quad\zeta(0)=\zeta_0\quad\text{on }\Omega,
\end{align}
where $\varphi_{\bs v}=\varphi(x,t)$ depends on $\bs v\in\V_p$, and the vector field $\bs a$ satisfies \eqref{zero} and \eqref{a:monot:forte} with $p=2$ and $\ell=d$.

It is well known that for each $\varphi\in L^2(Q_T)$ and $\zeta_0\in L^2(\Omega)$, the weak solution $\zeta\in L^\infty\big(0,T;L^2(\Omega)\big)\cap L^2\big(0,T;H^1_0(\Omega)\big)$  to \eqref{p1}, \eqref{p2}  and \eqref{iv} depends continuously, in these spaces, for the variation of $\varphi$ in the weak topology of $L^2(Q_T)$.
Moreover, if $\zeta_0\in \C^{\gamma}(\overline\Omega)$ is H\"older continuous for some $0<\gamma<1$ and $\varphi\in L^q(Q_T)$ for $q>\frac{d+2}{2}$, the following estimate holds (see \cite{LadyzhenskayaSolonnikovUraltseva1968},  p. 419)
\begin{equation*}
\|\zeta\|_{\C^\lambda (\overline Q_T)}\leq C\left(\|\zeta_0\|_{\C^\gamma(\overline\Omega)}+ \|\varphi\|_{L^q(Q_T)}\right),
\end{equation*}
for some $\lambda$, $0<\lambda\leq\gamma<1$, where $C>0$ is a constant independent of the data $\varphi$.

Now, for each $\bs v\in\V_p$ with $p>\frac{d+2}2$ ($p=2$ if $d=1$) and given $\bs\psi\in L^\infty(Q_T)^m$ and $\bs\eta\in L^\infty(Q_T)^\ell$, we may choose in \eqref{G0}, $\zeta=\zeta(\bs v)$ as being the solution of \eqref{p1}, \eqref{p2}, with a given $\zeta_0\in\C^\gamma(\overline\Omega)$ and
\begin{equation}\label{H}
\varphi_{\bs v}=\varphi_0+\bs\psi\cdot\bs v+\bs\eta\cdot L\bs v\in L^p(Q_T),
\end{equation}
for some fixed $\varphi_0\in L^p(Q_T).$ Hence, by \eqref{H} and Ascoli theorem, the mapping $\bs v\mapsto\varphi_{\bs v}\mapsto\zeta(\bs v)$ is completely continuous from $\V_p$ into $\C(\overline Q_T)$. Indeed, if $\bs v_n\rightharpoonup\bs v$ in $\V_p$, $\{\zeta(\bs v_n)\}_n$ is bounded in ${\C}^\lambda(\overline Q_T)\cap L^2\big(0,T;H^1_0(\Omega)\big)$ and, for some subsequence, $\zeta(\bs v_n)\rightharpoonup\zeta$ weakly in $L^2\big(0,T;H^1_0(\Omega)\big)$ and uniformly in $\overline Q_T$, for a $\zeta\in L^2\big(0,T;H^1_0(\Omega)\big)\cap\C(\overline Q_T)$, where we have $\zeta=\zeta(\bs v)$ by monotonicity and uniqueness of the solution of \eqref{p1}, \eqref{p2}. Then the whole sequence converges and the complete continuity of $G=G[\bs v]$, from $\V_p$ into $\C(\overline Q_T)$, is guaranteed by the assumptions.

We observe that, if $\X_p\subset W^{s,p}(\Omega)^m$, $s=1,2,\dots$, we can also choose in \eqref{p1}
\begin{equation*}
\varphi_{\bs v}=\varphi_0+\sum_{0\le|\alpha|\le s}\bs\psi_\alpha\cdot\partial^\alpha\bs v,
\end{equation*}
with $\bs\psi_\alpha\in L^\infty(Q_T)^m$, provided $p>\frac{d+2}{2}$, and even more general terms involving linear combinations of the gradients of the $\bs\psi_\alpha\cdot\partial^\alpha\bs v\in L^p(Q_T)$, $0\le|\alpha|\le s$, provided $p>d+2$.

\subsection{Linear differential operators}$ $

In this subsection we illustrate some concrete results for the operators $L$ referred as examples in the Introduction for convex sets of the type \eqref{kapa}. For simplicity, in all the examples we consider the vector fields
$$ \bs a(\bs\xi)=\alpha\,|\bs \xi|^{p-2}\bs\xi,\ \bs\xi\in\R^\ell,\qquad\alpha=\alpha(x,t)\ge0\ \text{a.e. in }Q_T\qquad\text{and}\qquad \bs b\equiv\bs0$$
and we assume that the operator $G$ satisfies the Assumption \ref{Phi}.

\subsubsection{\bf {\footnotesize A problem with gradient constraint}}$ $

\begin{coro}
Let $\Omega$ be a bounded open subset of $\R^d$ with a Lipschitz boundary, $\V_p=L^p\big(0,T;W^{1,p}_0(\Omega)\big)$ and  $p>\max\big\{1,\frac{2d}{d+2}\big\}$. Let $f\in L^2(Q_T)$ and $u_0\in\K_{G[u_0]}$. Then the following quasi-variational inequality	has a weak solution 
\begin{equation*}
\left\{\begin{array}{l}
u\in\K_{G[u]},\vspace{3mm}\\
\displaystyle\int_0^T\langle\partial_t v, v- u\rangle_p+\int_{Q_T}\alpha\,|\nabla u|^{p-2}\nabla u\cdot\nabla( v- u)
\ge\int_{Q_T} f( v- u)-\frac12\int_\Omega|v(0)-u_0|^2,\vspace{3mm}\\
\hfill{\forall\, v\in \Y_p\text{ such that }
v\in\K_{G[u]}}.
\end{array}
\right.
\end{equation*}
\hfill{$\square$}
\end{coro}

Indeed with $Lu=\nabla u$ and $V_p=W^{1,p}(\Omega)$, the Assumptions \ref{op:L} to \ref{Phi} are satisfied since the inclusion of $\X_p=W^{1,p}_0(\Omega)$ into $\HH=L^2(\Omega)$ is compact for $p>\max\big\{1,\frac{2d}{d+2}\big\}$.
 
The degenerate case $\alpha\equiv0$ corresponds to the variational model of sandpile growth where $G$ models the slope of the pile (see \cite{Pri1996-1}). In \cite{Pri1996-1}, Prigozhin introduces an operator $G$ which is discontinuous in the height $ u$ of the sandpile and leads to a quasi-variational formulation that is still an open problem.

\subsubsection{\bf {\footnotesize A problem with Laplacian constraint}}$ $

\begin{coro}	Let $\Omega$ be a bounded open subset of $\R^d$ with a $\C^{1,1}$ boundary, $\V_p=L^p(0,T;W^{2,p}_0(\Omega))$ with $p>\max\big\{1,\frac{2d}{d+4}\big\}$.  Let $f\in L^2(Q_T)$ and $u_0\in\K_{G[u_0]}$. Then the following quasi-variational inequality has a weak solution 
\begin{equation*}
\left\{\begin{array}{l}
u\in\K_{G[u]},\vspace{3mm}\\
\displaystyle\int_0^T\langle\partial_t v, v- u\rangle_p+\int_{Q_T} \alpha\,|\Delta u|^{p-2}\Delta u\,\Delta( v- u)
\ge\int_{Q_T} f( v- u)-\frac12\int_\Omega|v(0)-u_0|^2,\vspace{3mm}\\
\hfill{\forall\, v\in \Y_p\text{ such that }
v\in\K_{G[u]}.}
\end{array}
\right.
\end{equation*}
\hfill{$\square$}
\end{coro}	
	
Here we choose  $V_p=\{v\in L^p(\Omega):\Delta v\in L^p(\Omega)\}$, i.e. the operator $L$ is the Laplacian. The subspace $\X_p=W^{2,p}_0(\Omega)$ is endowed with the norm
$$\|v\|_{\X_p}=\|\Delta u\|_{ L^p(\Omega)}$$
which is  equivalent to the usual norm of $W^{2,p}(\Omega)$ because $\Delta$ is an isomorphism between 
$\X_p$ and $L^p(\Omega)$. Besides, $(\X_p, L^2(\Omega),  \X_p')$ is a Gelfand triple and the inclusion $\X_p\subset L^2(\Omega)$ is compact because $p>\max\big\{1,\frac{2d}{d+4}\big\}$. 

\subsubsection{\bf {\footnotesize A problem with curl constraint}}$ $

\begin{coro}
Let $\Omega$ be a bounded open subset of $\R^3$ with a Lipschitz boundary,	$p>\frac65$, $f\in L^2(Q_T)$. 
Define 
$$\X_p=\big\{\bs v\in L^p(\Omega)^3:\nabla\times\bs v\in L^p(\Omega)^3,
\nabla\cdot\bs v=0,\,\bs v\cdot\bs n_{_{|\partial\Omega}}=\,0\big\}$$
or
$$\X_p=\big\{\bs v\in L^p(\Omega)^3:\nabla\times\bs v\in L^p(\Omega)^3, \nabla\cdot\bs v=0,\,\bs v\times\bs n_{_{|\partial\Omega}}=\,\bs 0\big\}.$$
	
If $\bs u_0\in\K_{G[\bs u_0]}$, the following quasi-variational inequality	has a weak solution 
\begin{equation*}
\left\{\begin{array}{l} 
\bs u\in\K_{G[\bs u]},\vspace{3mm}\\
\displaystyle\int_0^T\langle\partial_t \bs v, \bs v- \bs u\rangle_p+\int_{Q_T}\alpha\,  |\nabla\times\bs u|^{p-2}\nabla\times\bs v\cdot\nabla\times(\bs v-\bs v)\vspace{3mm}\\
\hfill{\hspace{3cm}\ge\displaystyle\int_{Q_T}\bs  f\cdot(\bs v-\bs u)-\frac12\int_\Omega|\bs v(0)-\bs u_0|^2\quad\forall\,\bs  v\in \Y_p\text{ such that }
\bs	v\in\K_{G[\bs u]}.}
\end{array}
\right.
\end{equation*}
\hfill{$\square$}
\end{coro}

Here $L\bs v=\nabla\times\bs v$ and $\V_p=\{\bs v\in L^p(\Omega)^3:\nabla\times\bs v\in L^p(\Omega)^3\}$. In both choices of $\X_p$, corresponding to different boundary conditions, it is well known that $\X_p$ is a closed subspace of $W^{1,p}(\Omega)^3$ and that the semi-norm $\|\Rot\,\cdot\,\|_{L^p(\Omega)^3}$ is a norm equivalent to the one induced in $\X_p$ by the usual norm in $W^{1,p}(\Omega)^3$ (for details see \cite{AmroucheSeloula2010}). 
Here $\X_p$ is compactly embedded in $\HH=\{\bs v\in L^2(\Omega)^3:\nabla\cdot\bs v=0\}$.

This model is related to the Bean-type superconductivity variational inequality, which was solved in \cite{MirandaRodriguesSantos2012}, with prescribed critical threshold $G$. If we let here this threshold be, for instance, dependent on the temperature $\zeta$ defined by \eqref{p1}-\eqref{p2} and we impose $p>\frac52$, we obtain the existence of a weak solution to the corresponding thermal and electromagnetic coupled problem.

\subsubsection{\bf {\footnotesize Non-Newtonian thick fluids - a problem with a constraint on $ D$}}$ $
 
Denote 
$$ D\bs u=\tfrac12(\nabla \bs u+\nabla\bs u^T),$$
$$V_p=\big\{\bs v\in L^p(\Omega)^d:\, D\bs v\in L^p(\Omega)^{d^2}\big\},\qquad\mathbb J=\big\{\bs v\in \D(\Omega)^d:\nabla\cdot\bs v=0\big\}$$
and
$$\X_p=\overline{\mathbb J}^{W^{1,p}(\Omega)^d},\ \text{ for }p>1,\ d\ge2.$$
Let
$\V_p=L^p(0,T;V_p)$  and observe that $\X_p$ is compactly embedded in $\HH=\{\bs v\in L^2(\Omega)^d:\nabla\cdot\bs v=0\}=\overline{\mathbb J}^{L^2(\Omega)^d}$, if $p>\frac{2d}{d+2}$, by Sobolev and Korn inequalities. Hence, using the results of \cite{Rodrigues2014} and  \cite{MirandaRodrigues2016} for the variational inequality for incompressible thick fluids in the simpler case of the Stokes flow, we obtain the following conclusion:
\begin{coro} Let $\Omega$ be a bounded open subset of $\R^d$ with a Lipschitz boundary, $d\ge2$,
 $p>\frac{2d}{d+2}$,  $\bs f\in L^2(Q_T)^d$ and  $\bs u_0\in\K_{G[\bs u_0]}$. Then the following quasi-variational inequality
\begin{equation*}
\left\{\begin{array}{l} 
\bs u\in\K_{G[\bs u]},\vspace{3mm}\\
\displaystyle\int_0^T\langle\partial_t \bs v, \bs v- \bs u\rangle_p+\int_{Q_T}\alpha\,|D\bs u|^{p-2}  D\bs u\cdot D(\bs v-\bs u)	\ge\displaystyle\int_{Q_T}\bs  f\cdot(\bs v-\bs u)-\frac12\int_\Omega|\bs v(0)-\bs u_0|^2,\vspace{3mm}\\
\hfill{\forall\,\bs  v\in \Y_p\text{ such that }\bs	v\in\K_{G[\bs u]}}
\end{array}
\right.
\end{equation*}
has a weak solution.
\hfill{$\square$}
\end{coro}

\subsubsection{\bf {\footnotesize A problem with first order vector fields constraint}}$ $

Let $\Omega\subset \R^d$, $d\ge2,$ be a connected bounded open set and $L=(X_1,\ldots,X_\ell)$ be a family of Lipschitz vector fields on $\R^d$ that connect the space. We shall assume that the regularity of $\partial\Omega$ and the structure of $L$ support the following Sobolev-Poincaré compact embedding for $p\ge2$,
\begin{equation}\label{cp2}
\X_p\hookrightarrow L^2(\Omega).
\end{equation}

This is the case of an H\"ormander operator with
$$X_j=\sum_{i=1}^d\alpha_{ij}\partial_{x_i},\qquad j=1,\ldots,\ell,$$
with $\alpha_{ij}\in\C^\infty(\overline\Omega)$ such that the Lie algebra generated by these $\ell$ vector fields has dimension $d$, when the set $\X_p$ is the closure of $\D(\Omega)$ in
$$V_p=\big\{v\in L^p(\Omega): X_jv\in L^p(\Omega), j=1,\ldots,\ell\big\},\ \text{ with }p\ge 2,$$
with the graph norm and $\partial\Omega\in\C^\infty$. 
Indeed, in this case, it is known (see \cite{DerrigDias1972}, \cite{dd} and \cite{d91}) the following extension of the Rellich-Kondratchov theorem,
$$\X_p=\overline{\D(\Omega)}^{V_p}\hookrightarrow L^2(\Omega)\ \text{ is compact for }p\ge2,$$
holds and so $(\X_p,L^2(\Omega), \X_p')$ is a Gelfand triple with compact embeddings. For other classes of vector fields, namely associated with degenerate subelliptic operators, and a characterization of domains where \eqref{cp2} holds, see for instance \cite{dd} and \cite{cdg93}. By application of Theorem \ref{IQV0} we can now conclude the following existence result:

\begin{coro} Suppose that $\Omega$ is a bounded open subset of $\R^d$ with a smooth boundary.  Under the assumption \eqref{cp2}, if $p\ge2$, $f\in L^2(Q_T)$ and $u_0\in\K_{G[u_0]}$, the quasi-variational inequality
\begin{equation*}
\left\{\begin{array}{l}
u\in\K_{G[u]},\vspace{3mm}\\
\displaystyle\int_0^T\langle\partial_t v, v- u\rangle_p+\int_{Q_T}\alpha \sum_{j=1}^\ell\Big(\sum_{i=1}^\ell  |X_iu|^2\Big)^{\frac{p-2}2}X_ju\, X_j(v-u)\vspace{3mm}\\
\hspace{3cm}\hfill{\ge\displaystyle\int_{Q_T} f( v- u)-\frac12\int_\Omega|v(0)-u_0|^2},
\quad\forall\, v\in \Y_p\text{ such that }
v\in\K_{G[\bs u]}
\end{array}
\right.
\end{equation*}
has a weak solution.
\end{coro}

\section{The approximated problem}
\label{app-section}

In order to establish the existence of solution to the quasi-variational inequality~\eqref{iqv},  we start by proving existence of the solution 
to the problem of an approximated system of equations, defined for fixed $\bs\varphi\in\HHH$, $\delta\in(0,1)$ and $\eps\in(0,1)$. With this regularisation and penalisation of the quasi-variational inequality \eqref{iqv} with convex sets $\K_{G[\bs \varphi](t)}$, $t\in[0,T]$, we apply a fixed point argument. 
Consider the following increasing continuous function $k_{\eps}:\R\to\R^+_0$:
\begin{equation}\label{kapa_eps}
k_\varepsilon(s) = 
\begin{cases}
0 &\text{ if } s\leq0,\\
e^\frac{s}{\varepsilon}-1&\text{ if }0\leq s\leq \frac1\varepsilon,\\
e^\frac{1}{\varepsilon^2}-1 &\text{ if } s\geq\frac1\varepsilon.
\end{cases}
\end{equation}
Observe that the function $k_{\eps\delta}=\delta+k_\eps$ approximates the maximal monotone graph
$$\overline k_\delta(s)\in\begin{cases}
\{\delta\}&\text{ if }s<0,\\
[\delta,\infty[&\text{ if }s=0.
\end{cases}
$$

We start with an auxiliary lemma.

\begin{lemma}\label{monot*} Let $\psi$ be a scalar real function defined in $Q_T$. Then, for  $p\in(1,\infty)$, the operator 
\begin{equation}\label{monotone}
\bs T_{\eps}(x,t,\bs \xi)=k_\eps\big(|\bs \xi|-\psi(x,t)\big)|\bs \xi|^{p-2}\bs \xi
\end{equation}
is monotone.
\end{lemma}
\begin{proof} To simplify, we omit the argument $(x,t)$ and we denote $k_\eps(|\bs\xi|-\psi)$ simply by $k_\eps(|\bs\xi|)$.
We may assume, without loss of generality, that $|\bs\xi|\ge|\bs\xi'|$. Because $\bs S(\bs\xi)=|\bs\xi|^{p-2}\bs\xi$ is monotone and $k_\eps$ is a nonnegative increasing function, we have
\begin{multline*}
\big(\bs T_\eps(\bs\xi)-\bs T_\eps(\bs\xi')\big)\cdot(\bs\xi-\bs\xi')
=
\big(\,k_\eps(|\bs\xi|)\bs S(\bs\xi)-
k_\eps(|\bs\xi'|)\bs S(\bs\xi')\,\big)
\cdot(\bs\xi-\bs\xi')\\
=
k_\eps(|\bs\xi'|)\big(\bs S(\bs\xi)-
\bs S(\bs\xi')\big)\cdot(\bs\xi-\bs\xi')
+
\big(\, k_\eps(|\bs\xi|-k_\eps(|\bs\xi'|\,\big))|\bs\xi|^{p-2}\bs\xi\cdot(\bs\xi-\bs\xi')\\
\ge \big(\, k_\eps(|\bs\xi|-k_\eps(|\bs\xi'|)\,\big)\,|\bs\xi|^{p-1}\,(|\bs\xi|-|\bs\xi'|)\ge0.
\end{multline*}
\end{proof}

\begin{prop}\label{existence-iv} Suppose that Assumptions \ref{op:L} to  \ref{gelfand} are satisfied.
Considering functions 
\begin{equation}\label{assump1}
\bs f\in L^{2}(Q_T)^m,\quad \bs\varphi\in \HHH\quad\text{and}\quad \bs u_0\in \K_{G[\bs\varphi](0)},
\end{equation} 
the problem that consists of finding $\bs u_{\varepsilon\delta,\bs \varphi}$ such that
\begin{equation}\label{app}
\left\{\begin{array}{l}
\displaystyle\langle\partial_t\bs u_{\varepsilon\delta,\bs\varphi}(t),\bs\psi\rangle_p
+\int_\Omega \bs a(t,L\bs u_{\eps\delta,\bs\varphi}(t))\cdot L\bs\psi+\int_\Omega \bs b(t,\bs u_{\eps\delta,\bs\varphi}(t))\cdot \bs\psi\vspace{3mm}\\
\hspace{2cm}\hfill{+\displaystyle\int_\Omega\big(\delta+ k_\varepsilon\big(|L \bs u_{\eps\delta,\bs\varphi}(t)| -G[\bs\varphi](t)\big)\big) \left|L\bs u_{\eps\delta,\bs\varphi}(t)\right|^{p-2}L \bs u_{\eps\delta,\bs\varphi}(t)\cdot L\bs\psi}\vspace{3mm}\\
\hfill{\displaystyle=\int_\Omega\bs f(t)\cdot\bs\psi,\qquad\forall \bs\psi\in\X_p,\quad\text{for a.e. }t\in(0,T),}\vspace{3mm}\\
\bs u_{\eps\delta,\bs\varphi}(0) =\bs u_0
\end{array}
\right.
\end{equation}	
has a unique solution $\bs u_{\varepsilon\delta,\bs\varphi}\in \V_p$, with $\partial_t\bs u_{\eps\delta,\bs\varphi}\in\V_p'$, i.e.\ $\bs u_{\eps\delta,\varphi}\in\Y_p\subset\C\big([0,T];\HH\big)$.
\end{prop}
\begin{proof}
The existence and uniqueness of solution of problem \eqref{app} is consequence of a general result for parabolic quasilinear operators of monotone type (see, for instance, \cite[Theorem 8.9, p.~224 or Theorem 8.30, p.~243]{Roubicek2013}).
\end{proof}

\begin{prop} \label{propXp} Suppose that   Assumptions \ref{op:L} to \ref{gelfand} are satisfied.
Under the assumption~\eqref{assump1}, the solution $\bs u_{\varepsilon\delta,\bs\varphi}$ of the problem \eqref{app} verifies the following a priori estimates
\begin{align}
\label{est1}
\left\|\bs u_{\varepsilon\delta,\bs\varphi}\right\|_{L^\infty(0,T;L^2(\Omega)^m)}
& \leq C,\\
\label{est2}
\left\|L \bs u_{\varepsilon\delta,\bs\varphi}\right\|_{L^p(Q_T)^\ell}
&\le \tfrac{C}{\delta^\frac1p},\\
\label{est4}
\left\|\partial_t\bs u_{\varepsilon\delta,\bs\varphi}\right\|_{(\V_p)'}
& \leq \tfrac{C_\eps}{\delta^\frac1{p'}},
\end{align}
where $C$ and $C_\eps$ are positive constants independent of $\bs\varphi$ and of $\delta$. $C$ is also independent of $\eps$.
\end{prop}
\begin{proof}
Using $\bs w=\bs u_{\varepsilon\delta,\bs\varphi}$ as a test function in \eqref{app} 
we get, for a.e. $t\in(0,T)$,
\begin{multline*}
\langle\partial_t\bs w(t),\bs w(t)\rangle_p+\int_\Omega\bs a(t, L\bs w(t))\cdot L\bs w(t)
+\int_\Omega\bs b(t, \bs w(t))\cdot\bs w(t)\\
+\int_\Omega\Big(\delta+ k_\varepsilon\big(|L \bs w(t)| -G[\bs\varphi](t)\big)\Big)\left|L\bs w(t)\right|^p
=\int_\Omega\bs f(t)\cdot\bs w(t).
\end{multline*}
Denote $Q_t=\Omega\times(0,t)$. Integrating the last equality between $0$ and $t$, recalling the monotonicity of $\bs a$, $\bs b$ and of  $\bs T_{\eps}$ defined in \eqref{monotone}, after applying H\"older and Young inequalities to  the right-hand side of the above equation, we obtain the inequality
\begin{equation}\label{desig1}
\int_\Omega|\bs w(t)|^2
+2\delta\int_{Q_t}|L\bs w|^p
\leq
\|\bs w\|^2_{L^2(Q_t)^m}+\|\bs f\|^2_{L^{2}(Q_t)^m}
+\int_\Omega|\bs u_0|^2.
\end{equation}
By the Gronwall inequality we conclude that
$$\|\bs w\|_{L^\infty(0,T;L^2(\Omega))}\le e^T\big(\|\bs f\|^2_{L^{2}(Q_T)^m}
+\|\bs u_0\|_{L^2(\Omega)^m}^2\big)$$
and so we proved \eqref{est1}. From \eqref{desig1}  we immediately obtain \eqref{est2}.
	
Next we prove that, given $\bs\psi\in\X_p$,
\begin{equation}\label{bvp'}
\left|\int_{Q_T}\bs b(\bs w)\cdot\bs\psi\right|\le 
C\max\big\{\|\bs w\|_{L^p(Q_T)}^{p-1},\|\bs w\|_{L^\infty(0,T;L^2(\Omega))}\big\}\|\bs\psi\|_{\X_p},
\end{equation}
being $C$ a positive constant.
We notice that, by the Assumption \ref{Xp}, $\X_p\subset L^2(\Omega)^m$. So, there exists a positive constant $C$ such that, for all $\bs v\in\X_p$, we have
$$\|\bs v\|_{L^{2\vee p}(\Omega)^m}\le C\|\bs v\|_{\X_p}.$$

We split the proof in two cases.

i) $1<p<2$

\begin{align*}
\left|\int_{Q_T}\bs b(\bs w)\cdot\bs\psi\right|&\le b^*\int_0^T\int_\Omega|\bs w(t)||\bs \psi|\\
&\le b^*\int_0^T\|\bs w(t)\|_{L^2(\Omega)^m}\|\bs \psi\|_{L^2(\Omega)^m}\\
&\le C b^*T\|\bs w\|_{L^{\infty}(0,T;L^2(\Omega)^m)}\|\bs \psi\|_{\X_p}.
\end{align*}

ii) $p\ge 2$

\begin{align*}
\left|\int_{Q_T}\bs b(\bs w)\cdot\bs\psi\right|&\le b^*\int_{Q_T}|\bs w|^{p-1}|\bs \psi|\\
&\le b^*\int_0^T\|\bs w(t)\|_{L^p(\Omega)^m}^{p-1}\|\bs\psi\|_{L^p(\Omega)^m}\\
& \le C b^*T^\frac1p\|\bs w\|_{L^p(Q_T)^m}^{p-1}\|\bs\psi\|_{\X_p}.
\end{align*}

From the first equation of \eqref{app}, we conclude that
\begin{multline*}
\Big|\langle\partial_t\bs w(t),\bs\psi\rangle_p\big|\le C_1\Big((a^*+\delta+e^\frac1{\eps^2})\|L\bs w(t)\|_{L^p(\Omega)^\ell}^{p-1}\|L\bs\psi\|_{L^p(\Omega)^\ell}\\
+	\max\big\{\|\bs w\|_{L^p(Q_T)^m}^{p-1},\|\bs w\|_{L^\infty(0,T;L^2(\Omega)^m)}\big\}\|\bs\psi\|_{\X_p}+\|\bs f(t)\|_{L^{2}(\Omega)^m}\big)\|\bs\psi\|_{L^2(\Omega)^m}\Big)
\end{multline*}
and so, using again the Assumption \ref{Xp},
\begin{multline*}
\int_0^T\big|\langle\partial_t\bs w(t),\bs\psi\rangle_p\big|^{p'}dt\le C_2\Big((a^*+e^\frac1{\eps^2})^{p'}\|L\bs w\|_{L^p(Q_T)^\ell}^p\\
+\max\big\{\|\bs w\|_{L^p(Q_T)^m}^{p},\|\bs w\|^{p'}_{L^\infty(0,T;L^2(\Omega)^m)}\big\}+\|\bs f\|_{L^{2}(Q_T)^m}^{p'}\Big)\|\bs\psi\|_{\X_p}^{p'},
\end{multline*}
concluding now easily that
\begin{equation*}
\|\partial_t\bs w(t)\|_{\V_p'}^{p'}=\int_0^T\sup_{\|\bs\psi\|_{X_p}\le 1}\big|\langle\partial_t\bs w(t),\bs\psi\rangle_p\big|^{p'}dt\le \frac{C_\eps}{\delta}.
\end{equation*}
\end{proof}

\begin{prop}\label{prop2} 
Suppose that Assumptions \ref{op:L} to  \ref{gelfand}  are verified. Assuming also \eqref{assump1}, define the function  $S:\HHH\to\Y_p$ by $S(\bs \varphi)=\bs u_{\varepsilon\delta,\bs\varphi}$, where $\bs u_{\varepsilon\delta,\bs\varphi}$ is the unique solution of  problem \eqref{app}. Then $S$  is continuous.
\end{prop}
\begin{proof}
Let us consider a sequence ${\{\bs \varphi_n\}}_n$ converging to $\bs \varphi$ in $\HHH$. Denoting $\bs w_n=\bs u_{\varepsilon\delta,{\bs\varphi_n}}$ and $\bs w=\bs u_{\varepsilon\delta,{\bs\varphi}}$,
we need to prove that 
\begin{equation*} 
\bs w_n\underset n\longrightarrow \bs w\ \text { in } \V_p
\quad \text{and}\quad
\partial _t\bs w_n\underset n\longrightarrow  \partial_t\bs w\ \text { in } \V_p'.
\end{equation*}

The argument is standard but we present it here for the sake of completeness.
Both functions $\bs w_n$ and $\bs w$ solve \eqref{app} so, for any $\bs\psi\in\X_p$,
\begin{multline*}
\left\langle\partial_t\left(\bs w_n(t)-\bs w(t)\right),\bs\psi)\right\rangle_p
+\int_\Omega\big( \bs a(t,L\bs w_n(t))-\bs a(t,L\bs w(t)))\cdot L\bs \psi\\
+\int_\Omega\big( \bs b(t,\bs w_n(t))-\bs b(t,\bs w(t)))\cdot \bs \psi
+\delta\int_\Omega\big( |L\bs w_n(t)|^{p-2}L\bs w_n(t)-|L\bs w(t)|^{p-2}L\bs w(t)\big)\cdot L\bs \psi\\
+\int_\Omega\Big( k_\varepsilon\big(|L \bs w_n(t)| -G[\bs\varphi_n](t)\big) \left|L\bs w_n(t)\right|^{p-2}L \bs w_n(t)
\\
- k_\varepsilon\big(|L \bs w(t)| -G[\bs\varphi](t)\big) \left|L\bs w(t)\right|^{p-2}L \bs w(t)\Big)\cdot L\bs\psi
=0.
\end{multline*}
Replacing $\bs\psi$ by 
$\bs w_n(t)-\bs w(t)$ in the last expression and integrating it over $(0,t)$ we get
\begin{multline}\label{aste}
\frac12\int_{\Omega}\left|\bs w_n(t)-\bs w(t)\right|^2+\int_{Q_t}\big( \bs a(L\bs w_n)-\bs a(L\bs w)\big)\cdot L(\bs w_n-\bs w)\\
+\int_{Q_t}\big( \bs b(\bs w_n)-\bs b(\bs w))\cdot (\bs w_n-\bs w)
+\delta\int_\Omega\big( |L\bs w_n|^{p-2}L\bs w_n-|L\bs w|^{p-2}L\bs w\big)\cdot L(\bs w_n-\bs w)\\
+\int_{Q_t}  \Big(k_\varepsilon\big(|L \bs w_n| -G[\bs\varphi_n]\big)\left|L\bs w_n\right|^{p-2}L \bs w_n-k_\varepsilon\big(|L \bs w| -G[\bs\varphi_n]\big)\left|L\bs w\right|^{p-2}L \bs w\Big)\cdot L\left(\bs w_n-\bs w\right)\\
=\int_{Q_t} \Big(k_\varepsilon\big(|L \bs w| -G[\bs\varphi]\big)-k_\varepsilon\big(|L \bs w| -G[\bs\varphi_n]\big)\Big) \left|L\bs w\right|^{p-2}L \bs w\cdot L(\bs w_n-\bs w).
\end{multline}

Using the monotonicity of $\bs a$, $\bs b$ and the operator $\bs T_\eps$ defined in \eqref{monotone}, we can neglect the second, third and fifth terms of the inequality above.

In the case $p\geq2$,  we obtain,  applying H\"older and Young inequalities, and denoting by $D_p$ the constant related to the strongly monotone term in $\delta$ (see \eqref{a:monot:forte}),
\begin{align*}
\frac12\int_{\Omega}\big|\bs w_n(t)&-\bs w(t)\big|^2+
\delta D_p\int_{Q_t}\left|L\left(\bs w_n-\bs w\right)\right|^p\\
&\leq\int_{Q_t} \Big|k_\varepsilon\big(|L \bs w| -G[\bs\varphi]\big)-k_\varepsilon\big(|L \bs w| -G[\bs\varphi_n]\big)\Big| \left|L\bs w\right|^{p-1}
| L\left(\bs w_n-\bs w\right)|\\
&\leq \frac{C_1}{\delta^{p-1}}\int_{Q_t} \Big|k_\varepsilon\big(|L \bs w| -G[\bs\varphi]\big)-k_\varepsilon\big(|L \bs w| -G[\bs\varphi_n]\big)\Big|^{p'} \left|L\bs w\right|^{p}\\
&\hspace{5mm}+\frac{\delta}2\,D_p\int_{Q_t}| L\left(\bs w_n-\bs w\right)|^p,
\end{align*}
and therefore  we get
\begin{multline}\label{cont1}
\|\bs w_n-\bs w\|_{L^\infty(0,T;L^2(\Omega)^m)}^2+\delta\, D_p\|L(\bs w_n-\bs w)\|_{L^p(Q_T)^\ell}^p\\
\leq \frac{2 C_1}{\delta^{p-1}}\int_{Q_t} \Big|k_\varepsilon\big(|L \bs w|-G[\bs\varphi])\big) -k_\varepsilon\big(|L \bs w| -G[\bs\varphi_n])\big) \Big|^{p'}\left|L\bs w\right|^p.
\end{multline}

Consider now the case $1<p<2$. From \eqref{aste}, we get again
\begin{multline*}
\frac12\int_{\Omega}\left|\bs w_n(t)-\bs w(t)\right|^2
+\delta\int_{Q_t}\big( |L\bs w_n|^{p-2}L\bs w_n-|L\bs w|^{p-2}L\bs w\big)\cdot L(\bs w_n-L\bs w)\\
\le\int_{Q_t} \Big|k_\varepsilon\big(|L \bs w| -G[\bs\varphi]\big)-k_\varepsilon\big(|L \bs w| -G[\bs\varphi_n]\big)\Big| \left|L\bs w\right|^{p-1}|L(\bs w_n-\bs w)|
\end{multline*}
and, using also the coercive condition on $\delta$ (see \eqref{a:monot:forte}) and the  H\"older inverse inequality, we obtain
\begin{multline*}
\frac12\int_\Omega \big|\bs w_n(t)-\bs w(t)\big|^2+\delta D_p\left(\|L\bs w_n\|_{L^p(Q_T)^\ell }^p + \|L \bs w\|_{L^p(Q_T)^\ell }^p\right)^\frac{p-2}{p}\|L(\bs w_n-\bs w)\|_{L^p(Q_T)^\ell }^2\\
\le 
\int_{Q_t} \Big|k_\varepsilon\big(|L \bs w| -G[\bs\varphi]\big)
-k_\varepsilon\big(|L \bs w| -G[\bs\varphi_n]\big)\Big|\left|L\bs w\right|^{p-1}
\left|L(\bs w_n-\bs w)\right|.
\end{multline*}
Recalling, by \eqref{est2}, 
$$\|L \bs w_n\|_{L^p(Q_T)^\ell }^p\le\frac{C^p}{\delta},\quad \|L \bs w\|_{L^p(Q_T)^\ell }^p\le\frac{C^p}{\delta}.$$
and	applying H\"older and Young inequalities to the right-hand side, we obtain
\begin{multline*}
\|\bs w_n-\bs w\|_{L^\infty(0,T;L^2(\Omega)^m)}^2+\delta^\frac2{p}C_2\|L(\bs w_n-\bs w)\|_{L^p(Q_T)^\ell}^2
\\
\leq \frac{1}{2\, C_2\,\delta^\frac{2}{p}}\left(\int_{Q_T} \Big|k_\varepsilon\big(|L\bs w| -G[\bs\varphi]\big)
-k_\varepsilon\big(|L\bs w_n| -G[\bs\varphi_n]\big)\Big|^{p'}\left|L\bs w\right|^p\right)^\frac{2}{p'}
+\frac{C_2\delta^\frac2{p}}2\,\|L(\bs w_n-\bs w)\|_{L^p(Q_T)^\ell }^2
\end{multline*}
and so 
\begin{multline}\label{cont2}
\|\bs w_n-\bs w\|_{L^\infty(0,T;L^2(\Omega)^m)}^2+\frac{C_2\delta^\frac2{p}}2\|L(\bs w_n-\bs w)\|_{L^p(Q_T)^\ell}^2\\
\leq \frac{1}{2\, C_2\,\delta^\frac{2}{p}}\left(\int_{Q_T} \Big|k_\varepsilon\big(|L\bs w| -G[\bs\varphi]\big)
-k_\varepsilon\big(|L\bs w_n| -G[\bs\varphi_n]\big)\Big|^{p'}\left|L\bs w\right|^p\right)^\frac{2}{p'}.
\end{multline}
Observe now that we have a.e. in $Q_T$
\begin{equation*}
\Big|k_\varepsilon\big(|L\bs w| -G[\bs\varphi]\big)
-k_\varepsilon\big(|L\bs w_n| -G[\bs\varphi_n]\big)\Big|^{p'}\left|L\bs w\right|^p
\leq 
\left(2e^\frac{1}{\varepsilon^2}\right)^{p'}\left|L\bs w\right|^p
\end{equation*}
and $k_\varepsilon$ is a continuous function. Recalling that $\bs\varphi_n\underset{n}{\longrightarrow}\bs\varphi$ in $\HHH$,  the Assumption \ref{Phi} implies that $G[\bs \varphi_n]\underset{n}{\longrightarrow}G[\bs\varphi]$ in $L^1(Q_T)$ and so, at least for a subsequence, 
$$G[\bs\varphi_n]\underset n\longrightarrow G[\bs\varphi]\quad \text{a.e.\ in } Q_T$$
and, by the dominated convergence theorem
$$k_\eps(|L\bs w_n|-G[\bs\varphi_n])\underset{n}{\longrightarrow}k_\eps(|L\bs w|-G[\bs\varphi])\quad\text{in }L^{p'}(Q_T)$$
and so the right-hand sides of \eqref{cont1} and \eqref{cont2} converge to zero a.e., when   $n\to\infty$. 

By definition,
\begin{align*}
\left\|\partial_t\left(\bs w_n-\bs w\right)\right\|_{\V_p'}^{p'}
&=\int_0^T\Big(\sup_{\|\bs\psi\|_{\X_p}\leq1}\left\langle\partial_t\left(\bs w_n(t)-\bs w(t)\right),\bs\psi)\right\rangle_p\Big)^{p'}\\
&=\int_0^T\Big(\sup_{\|\bs\psi\|_{\X_p}\leq1}\Big(\int_\Omega \big(\bs a(t,L\bs w(t))-\bs a(t,L\bs w_n(t))\big)\cdot L\bs\psi\\
&\qquad+\int_\Omega \big(\bs b(t,\bs w(t))-\bs b(t,\bs w_n(t))\big)\cdot \bs\psi\\
&\qquad+\delta\int_\Omega\big(|L\bs w(t)\big|^{p-2}L \bs w(t)-|L\bs w_n(t)\big|^{p-2}L \bs w_n(t)\big)\cdot L\bs\psi\\
&\qquad+\int_\Omega k_\varepsilon\big(|L \bs w(t)| -G[\bs\varphi](t)\big)\big|L\bs w(t)\big|^{p-2}L \bs w(t)\cdot L\bs\psi\\
&\qquad-\int_\Omega k_\varepsilon\big(|L \bs w_n(t)| -G[\bs\varphi_n](t)\big)\big|L\bs w_n(t)\big|^{p-2}L \bs w_n(t)\cdot L\bs\psi\Big)\Big)^{p'}.
\end{align*}
But
\begin{align*}
\Big(	\int_\Omega\big(\bs b(t,\bs w(t))-\bs b(t,\bs w_n(t))\big)\cdot \bs\psi\Big)^{p'}
&\le\Big(	\|\big(\bs b(t,\bs w(t))-\bs b(t,\bs w_n(t))\big)\|_{L^2(\Omega)^m}\|\bs\psi\|_{L^2(\Omega)^m}\Big)^{p'}\\
&\le C\Big(	\|\big(\bs b(\bs w)-\bs b(\bs w_n)\big)\|_{L^\infty(0,T;(L^2(\Omega)^m)}\|\bs\psi\|_{\X_p}\Big)^{p'}.
\end{align*}
Applying H\"older inequality, we conclude that
\begin{multline*}
\left\|\partial_t\left(\bs w_n-\bs w\right)\right\|_{\V_p'}
\le C\Big(\,\Big(\int_{Q_T} \big|\bs a(L\bs w)-\bs a(L\bs w_n)\big|^{p'}\Big)^\frac1{p'} \\
+\|\big(\bs b(\bs w)-\bs b(\bs w_n)\big)\|_{L^\infty(0,T;(L^2(\Omega)^m)}+\delta\Big(\int_{Q_T} \big||L\bs w_n|^{p-2}L\bs w_n-|L\bs w|^{p-2}L\bs w\big|^{p'}\Big)^\frac1{p'}\\
+\Big(\int_{Q_T}\left|k_\varepsilon\big(|L \bs w| -G[\bs\varphi]\big)
\big|L\bs w\big|^{p-2}L \bs w-k_\varepsilon\big(|L \bs w_n| -G[\bs\varphi_n]\big)\big|L\bs w_n\big|^{p-2}L \bs w_n\right|^{p'}\Big)^\frac1{p'}\Big)
\end{multline*}
and, arguing as before, we conclude the proof.
\end{proof}

\begin{thm}\label{app-iqv1} Suppose that  Assumptions \ref{op:L} to \ref{Phi} and \eqref{assump1} are verified.
Let $i$ be the inclusion of $\Y_p$ into $\HHH$ and $S:\HHH\to\Y_p$ the function defined in Proposition~\ref{prop2}.
Then the function $i\circ S$ has a fixed point in $\HHH$.
This fixed point solves the problem that consists on finding $\bs u_{\varepsilon\delta}\in\Y_p$ such that
\begin{equation}\label{app-iqv}
\left\{\begin{array}{l}
\displaystyle\langle\partial_t\bs u_{\eps\delta}(t),\bs\psi\rangle_p+\int_\Omega \bs a(t,L \bs u_{\eps\delta}(t))\cdot L\bs\psi
+\int_\Omega \bs b(t,\bs u_{\eps\delta}(t))\cdot \bs\psi\\
\hspace{2cm}+\displaystyle\delta\int_\Omega|L\bs u_{\eps\delta}(t)|^{p-2}L\bs u_{\eps\delta}(t)\cdot L\bs\psi\\
\hspace{2cm}+\displaystyle\displaystyle\int_\Omega k_\varepsilon\big(|L \bs u_{\eps\delta}(t)| -G[\bs u_{\eps\delta}(t)]\big) \left|L\bs u_{\eps\delta}(t)\right|^{p-2}L \bs u_{\eps\delta}(t)\cdot L\bs\psi\\
\hspace{2cm}=\displaystyle\int_\Omega\bs f(t)\cdot\bs\psi\qquad\forall\bs\psi\in\X_p\\
\bs u_{\eps\delta}(0) =\bs u_0.
\end{array}
\right.
\end{equation}
\end{thm}
\begin{proof} 
We use the Schauder fixed point theorem.
We already proved the continuity of $S$ and by the Assumption~\ref{gelfand} and Aubin-Lions Lemma embedding $\Y_p\hookrightarrow\HHH$ is compact, for $1<p<\infty$, and so $i\circ S$ is completely continuous as a map of $\HHH$ into itself.
By  Proposition~\ref{propXp}, given $\bs\varphi\in\HHH$, we have
$\|\bs u_{\eps\delta,\bs\varphi}\|_{\V_p}\le \frac{C}{\delta^\frac1p}$, where $C$ is a constant independent of $\bs\varphi$ and $\delta$ (it may depend on $\eps$).
Because $i$ is continuous, there exists $C_1$  such that $\|\bs v\|_{\HHH}\le C_1\|\bs v\|_{\Y_p}$ and we get 
$$\|i\circ S(\bs\varphi)\|_{\HHH}\le C\,C_1.$$
Then the image of $i\circ S$ is bounded, so we may apply Schauder fixed point theorem, obtaining immediately the conclusion of the existence of a $\bs u_{\eps\delta}=i\circ S(\bs u_{\eps\delta})$ in $\Y_p$.
\end{proof}

\section{Weak solutions of the quasi-variational inequality}

In this section, we prove the existence of a solution of the quasi-variational inequality \eqref{iqv} by taking suitable subsequences of solutions of \eqref{app-iqv} first as $\eps\to0$ and then as $\delta\to0$.

Firstly we collect the {\em a priori} estimates for the solution $\bs u_{\eps\delta}$ of problem \eqref{app-iqv} which are independent of $\eps$. 

\begin{prop} \label{est} Suppose that Assumptions \ref{op:L} to  \ref{Phi} are verified.
Assume also that $\bs f \in L^2(Q_T)^m$ and $\bs u_0\in \K_{[G(\bs u_0)]}$. Let  $\bs u_{\eps\delta}$ be a solution of the approximated problem \eqref{app-iqv}. Then there exists a positive constant $C$, independent of $\eps$ and $\delta$, such that
\begin{align}
\label{ueC}&\|\bs u_{\eps\delta}\|_{L^\infty(0,T;L^2(\Omega)^m)}\le C,\\
\label{ueV}&\|\bs u_{\eps\delta}\|_{\V_p}\le \frac{C}{\delta^\frac1p},\\
\label{aLue}&\|\bs a(L\bs u_{\eps\delta})\|_{\V_p'}\le \frac{C}{\delta^\frac1{p'}},\\
\label{cue}&\|\bs b(\bs u_{\eps\delta})\|_{\V_p'}\le \frac{C}{\delta^\frac1{p'}},\\
\label{3}&\|k_\eps(|L \bs u_{\eps\delta}|-G[\bs u_{\eps\delta}])\|_{L^1(Q_T)}\le C.
\end{align}
\end{prop}
\begin{proof}		
The first two estimates are direct consequences of the inequalities \eqref{est1} and \eqref{est2} respectively, taking $\bs\varphi=\bs u_{\eps\delta}$.

Given $\bs\psi\in\V_p$, we have
\begin{align*}
\left|\int_{Q_T}\bs a(L\bs u_{\eps\delta})\cdot L\bs\psi\right|&\le a^*\int_{Q_T}|L\bs u_{\eps\delta}|^{p-1}|L\bs\psi|\\
&\le a^*\|L\bs u_{\eps\delta}\|_{L^p(Q_T)^\ell}^{p-1}\|L\bs\psi\|_{L^p(Q_T)^\ell}\\
&\le a^*\frac{C^{p-1}}{\delta^\frac1{p'}}\|L\bs\psi\|_{\V_p},
\end{align*}
proving \eqref{aLue}. 

From \eqref{bvp'}, we have
\begin{equation*}
\|\bs b(\bs w)\|_{\V_p'}\le 
C\max\big\{\|\bs w\|_{L^p(Q_T)^m}^{p-1},\|\bs w\|_{L^\infty(0,T;L^2(\Omega)^m)}\big\}.
\end{equation*}
But $\|\bs w\|_{L^\infty(0,T;L^2(\Omega)^m)}$ is uniformly bounded by \eqref{ueC} and
$$\|\bs w\|_{L^p(Q_T)^m}^{p-1}\le C_1 \|L\bs w\|_{L^p(Q_T)^m}^{p-1}\le \frac{C_2}{\delta^\frac1{p'}},$$
by \eqref{ueV}.

Choosing $\bs u_{\eps\delta}$ as test function in \eqref{app-iqv}, and integrating between $0$ and $t$, we get
\begin{multline*} 
\frac12\int_\Omega|\bs u_{\eps\delta}(t)|^2+\int_{Q_t}\bs a(L\bs u_{\eps\delta})\cdot L\bs u_{\eps\delta}+
\int_{Q_t}\bs b(\bs u_{\eps\delta})\cdot\bs u_{\eps\delta}\\
+\delta\int_{Q_t}|L\bs u_{\eps\delta}|^p+\int_{Q_t} k_\eps\big(|L \bs u_{\eps\delta}|-G[\bs u_{\eps\delta}]\big)|L\bs u_{\eps\delta}|^p= 
\int_{Q_t}\bs f\cdot\bs u_{\eps\delta}+\frac12\int_\Omega\bs u_0^2
\end{multline*}
and so, because $\bs a$ and $\bs b$ are monotone, $\bs a(\bs 0)=\bs 0$, $\bs b(\bs 0)=\bs 0$ and using the Gronwall inequality, we obtain
\begin{equation}\label{ajuda}
\int_{Q_t}k_\eps\big(|L \bs u_{\eps\delta}|-G[\bs u_{\eps\delta}]\big)|L\bs u_{\eps\delta}|^p\le C_1,
\end{equation}
where $C_1$ is a constant independent of $\eps$ and $\delta$.
As $G[\bs u_{\eps\delta}]\ge g_*>0$ and $k_\eps\equiv0$ in $\{|L\bs u_{\eps\delta}|\le G[\bs u_{\eps\delta}]\}$, then
\begin{multline*}
\int_{Q_T}k_\eps\big(|L \bs u_{\eps\delta}|-G[\bs u_{\eps\delta}]\big)=\int_{\{|L\bs u_{\eps\delta}|>G[\bs u_{\eps\delta}]\} }k_\eps\big(|L \bs u_{\eps\delta}|-G[\bs u_{\eps\delta}]\big)\\
\le\int_{\{|L\bs u_{\eps\delta}|>G[\bs u_{\eps\delta}]\} }k_\eps\big(|L \bs u_{\eps\delta}|-G[\bs u_{\eps\delta}])\big)\frac{|L\bs u_{\eps\delta}|^p}{(g_*)^p}
\end{multline*}
and, using \eqref{ajuda},
\begin{equation*}
	\|k_\eps\big(|L \bs u_{\eps\delta}|-G[\bs u_{\eps\delta}]\big)\|_{L^1(Q_T)}\le\frac{C_1}{(g_*)^p},
\end{equation*}
concluding the proof of \eqref{3}.
\end{proof}

\begin{lemma} 
\label{conv}
Let  $\bs u_{\eps\delta}$ be a solution of the approximated problem \eqref{app-iqv}. If $\bs u_\delta$ is the weak limit of a subsequence of $\{\bs u_{\eps\delta}\}_\eps$ when $\eps\to0$ then $\bs u_\delta\in \K_{G[\bs u_\delta]}$.
\end{lemma}
\begin{proof}
To prove that $\bs u_\delta$ belongs to the convex set $\K_{G[\bs u_\delta]}$, we use arguments as in \cite{MirandaRodriguesSantos2012}, which we adapt to our problem. We split $Q_T$ in three sets:
\begin{align}\label{ABC}
\nonumber A_{\eps\delta}&=\big\{(x,t)\in Q_T: \big|L\bs u_{\eps\delta}(x,t)\big|-G[\bs u_{\eps\delta}](x,t)<\sqrt{\varepsilon}\big\},\\
\nonumber B_{\eps\delta}&=\big\{(x,t)\in Q_T:\sqrt\eps\le\big|L\bs u_{\eps\delta}(x,t)\big|-G[\bs u_{\eps\delta}](x,t)<\tfrac1\varepsilon\big\},\\
C_{\eps\delta}&=\big\{(x,t)\in Q_T:\big|L\bs u_{\eps\delta}(x,t)\big|-G[\bs u_{\eps\delta}](x,t)\ge\tfrac1\varepsilon\big\}.
\end{align}
We recall that, by Assumption \ref{Phi}, the operator $G$ is compact. So, as a subsequence of $\{\bs u_{\eps\delta}\}_\eps$ (still denoted by $\{\bs u_{\eps\delta}\}_\eps$) converges weakly  to $\bs u_\delta$ in $\V_p$, then $G[\bs u_{\eps\delta}]$ converges to $G[\bs u_\delta]$ strongly in $\C\big([0,T];L^\infty(\Omega)\big)$ and
\begin{align*}
\int_{Q_T}\big(|L\bs u_\delta|-G[\bs u_\delta]\big)^+&=\int_{Q_T}\varliminf_{\eps\rightarrow0}\Big(\big(|L\bs u_{\eps\delta}|- G[\bs u_{\eps\delta}]\big)\wedge\frac1\eps\vee\sqrt\eps\Big)\\
&\le\varliminf_{\eps\rightarrow0}\int_{Q_T}\Big(\big(|L\bs u_{\eps\delta}|- G[\bs u_{\eps\delta}]\big)\wedge\frac1\eps\vee\sqrt\eps\Big)
\\
&=\varliminf_{\eps\rightarrow0}\Big(\int_{A_{\eps\delta}}\sqrt\eps+
\int_{B_{\eps\delta}}\big(|L\bs u_{\eps\delta}|- G[\bs u_{\eps\delta}]\big)+\int_{C_{\eps\delta}}\frac1\eps\Big).
\end{align*}
We observe that 
\begin{equation*}
\int_{A_{\eps\delta}}\sqrt\eps\le\sqrt\eps|Q_T|\underset{\varepsilon\to0}{\longrightarrow}0,
\end{equation*}
\begin{equation*}
|B_{\eps\delta}|=\int_{B_{\eps\delta}}1
\le\int_{Q_T}\frac{k_\eps(|L\bs u_{\eps\delta}|- G[\bs u_{\eps\delta}])+1}{e^\frac1{\sqrt\eps}} \leq \frac{C}{e^\frac1{\sqrt\eps}} \underset{\varepsilon\to0}{\longrightarrow}0
\end{equation*}
and
\begin{equation*}
\int_{C_{\eps\delta}}\frac1\eps
\le\frac1\eps\int_{Q_T}\frac{k_\eps(|L\bs u_{\eps\delta}|- G[\bs u_{\eps\delta}])}{e^\frac1{\eps^2}-1}\leq \frac{C}{\eps(e^\frac1{\eps^2}-1)}\underset{\varepsilon\to0}{\longrightarrow}0.
\end{equation*}
Besides,
$$\int_{B_{\eps\delta}}\big(|L\bs u_{\eps\delta}|- G[\bs u_{\eps\delta}]\big)\le\||L\bs u_{\eps\delta}|- G[\bs u_{\eps\delta}]\|_{L^p(Q_T)}\,|B_{\eps\delta}|^\frac1{p'}\underset{\varepsilon\to0}{\longrightarrow}0.$$
So
$$\int_{Q_T}\big(|L\bs u_\delta|-G[\bs u_\delta]\big)^+=0,$$
which means that 
$$|L\bs u_\delta|\le G[\bs u_\delta]\quad\text{a.e. in }Q_T.$$
\end{proof}

\begin{lemma}\label{RegSeq}
Let  $\bs v\in \V_p=L^p(0,T;\X_p)$ be such that $\bs v\in\K_{G[\bs v]}$, and $\bs z\in\K_{G[\bs v](0)}$. 
Then there exists a regularizing sequence $\{\bs v_n\}_n$ and a sequence of scalar functions $\{g_n\}_n$ with the following properties:

i) $\bs v_n\in L^\infty(0,T;\X_p)$ and $\partial_t\bs v_n\in L^\infty(0,T;\X_p)$;

ii) $\bs v_n\underset{n}{\longrightarrow}\bs v$ in $\V_p$ strongly;

iii) $\displaystyle\varlimsup_n\int_0^T\langle\partial_t\bs v_{ n}, \bs v_{ n}-\bs v\rangle_p\le0$;

iv) $|L\bs v_n|\le g_n$, where $g_n\in\C\big([0,T];L^\infty(\Omega)\big)$ and $g_n\underset{n}{\longrightarrow}G[\bs v]$ in $\C\big([0,T];L^\infty(\Omega)\big)$.
\end{lemma}
\begin{proof}
Let $\bs v_{ n}$ be the unique solution of the ordinary differential equation $\bs v_n+\frac1n\partial_t\bs v_n=\bs v$, with $\bs v_n(0)=\bs z$. 
The function $\bs v_n$ has the following expression:
$$\bs v_n(x,t)=e^{-nt}\int_0^t\bs v(x,\tau)ne^{n\tau}d\tau+e^{-nt}\bs z(x)$$
and it is well known it satisfies i), ii) and iii) (see \cite[p.~274]{Lions1969} or \cite[p.~206]{Roubicek2013}).
Therefore, it follows
\begin{multline*}
|L\bs v_n(x,t)|\le e^{-nt}\int_0^t|L\bs v(x,\tau)|ne^{n\tau}d\tau
+e^{-nt}|L\bs z(x)|\\
\le e^{-nt}\int_0^t G[\bs v](x,\tau)ne^{n\tau}d\tau+e^{-nt}G[\bs v](x,0),\quad\text{for a.e. }(x,t)\in Q_T.
\end{multline*}
Defining
$$
g_n(x,t)= e^{-nt}\int_0^t G[\bs v](x,\tau)ne^{n\tau}d\tau+e^{-nt}G[\bs v](x,0)
=\int_{-\infty}^0\widetilde G[\bs v](x,\tfrac sn+t)e^sds,
$$
where $\widetilde G[\bs v](x,t)$ denotes the extension of $G[\bs v]\in\C\big([0,T];L^\infty(\Omega)\big)$ by $G[\bs v](x,0)$ for $t<0$, we have 
\begin{equation*}
\big|g_n(x,t)-G[\bs v](x,t)|\le\int_{-\infty}^0|\widetilde G[\bs v](x,\tfrac{s}{n}+t)-G[\bs v](x,t)\big|e^sds\\
\end{equation*}
and, by the uniform continuity of $G[\bs v](t)$ in $[0,T]$ with values in $L^\infty(\Omega)$, we have also the uniform convergence of
$g_n\underset{n}{\rightarrow}G[\bs v]$ in $\C\big([0,T];L^\infty(\Omega\big)$, concluding iv).
\end{proof}

\begin{proof}[Proof of Theorem \ref{IQV0}]  The boundedness of $\{\bs u_{\eps\delta}\}_\eps$ in $L^\infty\big(0,T;L^2(\Omega)^m\big)\cap\V_p$ implies that there exists a subsequence, still denoted by $\{\bs u_{\eps\delta}\}_\eps$, converging  to a function $\bs u_\delta$ when $\eps\rightarrow0$, in $L^\infty\big(0,T;L^2(\Omega)^m\big)\cap\V_p$ weak-$*$. 
So, by the Lemma~\ref{conv}, $\bs u_\delta\in\K_{G[\bs u_\delta]}$ and we may extract a subsequence of $\{\bs u_\delta\}_\delta$ converging to some $\bs u$ weakly-$*$ in $L^\infty\big(0,T;L^2(\Omega)^m\big)\cap\V_p$ when $\delta\rightarrow0$. 

Observe that, for any measurable set $\omega\subset Q_T$,
$$\int_\omega|L\bs u|\le\varliminf_{\delta\rightarrow0}\int_\omega|L\bs u_\delta|\le\int_\omega\lim_{\delta\rightarrow0} G[\bs u_\delta]\le\int_\omega G[\bs u],$$
using the Assumption \ref{Phi}. Consequently, 
$$|L\bs u|\le G[\bs u]\text{ a.e. in }Q_T$$
and so $\bs u\in\K_{G[\bs u]}$.

{\bf Step 1:} The limit when $\eps\to0$

From the estimates, \eqref{aLue}, \eqref{cue} and \eqref{ueV} in Proposition \ref{est}, there exist $\bs\chi_\delta\in \V_p'$, $\bs\Upsilon_\delta\in  \V_p'$ and $\bs\Lambda_\delta\in L^{p'}(Q_T)^\ell$, such that, for subsequences, 
\begin{align}\label{tende}
\nonumber\bs a(L\bs u_{\eps\delta})&\underset{\eps\rightarrow0}{\lraup}\bs \chi_\delta \quad\text{in}\quad \V_p'\text{ weak},\\
\nonumber\bs b(\bs u_{\eps\delta})&\underset{\eps\rightarrow0}{\lraup}\bs \Upsilon_\delta \quad\text{in}\quad \V_p'\text{ weak}.
\\
|L\bs u_{\eps\delta}|^{p-2}L\bs u_{\eps\delta}&\underset{\eps\rightarrow0}{\lraup}\bs \Lambda_\delta \quad\text{in}\quad L^{p'}(Q_T)^\ell\text{ weak}.	
\end{align}

Define the operator $\A_\delta:\V_p\rightarrow\V_p'$ by
\begin{equation}\label{A}
	\langle \A_\delta\bs v,\bs w\rangle=\int_{Q_T}\big(\bs a(L\bs v)+\delta|L\bs v|^{-2}L\bs v\big)\cdot L\bs w+\int_{Q_T}\bs b(\bs v)\cdot\bs w.
\end{equation}

Given $\bs v$ belonging to the space $\Y_p$ defined in \eqref{embed}, we have
\begin{multline*}
\nonumber\int_0^T\langle\partial_t\bs u_{\eps\delta} ,\bs v -\bs u_{\eps\delta} \rangle_p=-\frac12\int_\Omega|\bs u_{\eps\delta}(T)-\bs v(T)|^2+\frac12\int_\Omega|\bs u_0-\bs v(0)|^2	+\int_0^T\langle\partial_t\bs v ,\bs v -\bs u_{\eps\delta} \rangle_p\\
\le \frac12\int_\Omega|\bs u_0-\bs v(0)|^2+\int_0^T\langle\partial_t\bs v,\bs v-\bs u_{\eps\delta}\rangle_p.
\end{multline*}

From now on, we denote $k_\eps(|L\bs u_{\eps\delta}-G[\bs u_{\eps\delta}])$ simply by $k_\eps$, with no risk of confusion.

Using $\bs u_{\eps\delta}-\bs v$ as test function in \eqref{app-iqv} and integrating between $0$ and $T$,
\begin{multline}\label{qquase}
\langle\A_\delta\bs u_{\eps\delta}, \bs u_{\eps\delta}\rangle\le \frac12\int_\Omega|\bs u_0-\bs v(0)|^2+\int_0^T\langle\partial_t\bs v,\bs v-\bs u_{\eps\delta}\rangle_p+\langle\A_\delta\bs u_{\eps\delta}, \bs v\rangle\\
+\int_{Q_T}k_\eps|L\bs u_{\eps\delta}|^{p-2}L \bs u_{\eps\delta}\cdot L(\bs v-\bs u_{\eps\delta})-\int_{Q_T}\bs f\cdot(\bs v-\bs u_{\eps\delta}).
\end{multline}
and so, for all $\bs v\in\Y_p$,
\begin{multline}\label{quase}
	\langle\A_\delta\bs u_{\eps\delta}, \bs u_{\eps\delta}-\bs u\rangle\le \frac12\int_\Omega|\bs u_0-\bs v(0)|^2+\int_0^T\langle\partial_t\bs v,\bs v-\bs u_{\eps\delta}\rangle_p+\langle\A_\delta \bs u_{\eps\delta}, \bs v-\bs u\rangle\\
	+\int_{Q_T}k_\eps|L\bs u_{\eps\delta}|^{p-2}L \bs u_{\eps\delta}\cdot L(\bs v-\bs u_{\eps\delta})
	-\int_{Q_T}\bs f\cdot(\bs v-\bs u_{\eps\delta}).
\end{multline}

Let $\bs u_n$ be the regularizing sequence of $\bs u$, defined in the previous lemma. Using $\bs u_n$ as test function in \eqref{quase}, we get
\begin{multline}\label{gn}
\langle\A_\delta\bs u_{\eps\delta}, \bs u_{\eps\delta}-\bs u\rangle\le \int_0^T\langle\partial_t\bs u_n,\bs u_n-\bs u_{\eps\delta}\rangle_p+\langle\A_\delta \bs u_{\eps\delta}, \bs u_n-\bs u\rangle\\
+\int_{Q_T}k_\eps|L\bs u_{\eps\delta}|^{p-2}L \bs u_{\eps\delta}\cdot L(\bs u_n-\bs u_{\eps\delta})
-\int_{Q_T}\bs f\cdot(\bs u_n-\bs u_{\eps\delta}).
\end{multline}
But
\begin{multline}\label{qqq}
\int_{Q_T}k_\eps |L\bs u_{\eps\delta}|^{p-2}L\bs u_{\eps\delta}\cdot L(\bs u_n-\bs u_{\eps\delta})\le
\int_{Q_T}k_\eps |L\bs u_{\eps\delta}|^{p-1}(|L\bs u_n|-|L\bs u_{\eps\delta}|)\\
=\int_{Q_T}k_\eps |L\bs u_{\eps\delta}|^{p-1}(|L\bs u_n|-G[\bs u_{\eps\delta}])
+\int_{Q_T}k_\eps |L\bs u_{\eps\delta}|^{p-1}(G[\bs u_{\eps\delta}]-|L\bs u_{\eps\delta}|)\\
\le\int_{Q_T}k_\eps |L\bs u_{\eps\delta}|^{p-1}(|L\bs u_n|-G[\bs u_{\eps\delta}]).
\end{multline}
In fact, the term $k_\eps |L\bs u_{\eps\delta}|^{p-1}(G[\bs u_{\eps\delta}]-|L\bs u_{\eps\delta}|)$ is less or equal to zero, because when $|L\bs u_{\eps\delta}|< G[\bs u_{\eps\delta}]$ then $k_\eps(|L\bs u_{\eps\delta}|-G[\bs u_{\eps\delta}]) =0$. 

Then, recalling \eqref{3} and \eqref{ajuda}, we conclude that
\begin{multline}\label{last}
\int_{Q_T}k_\eps |L\bs u_{\eps\delta}|^{p-1}(|L\bs u_n|-G[\bs u_{\eps\delta}])\le\int_{Q_T}k_\eps |L\bs u_{\eps\delta}|^{p-1}(G_n-G[\bs u_{\eps\delta}])\\
\le\Big(\int_{Q_T}k_\eps\Big)^\frac1p\Big(\int_{Q_T}k_\eps|L\bs u_{\eps\delta}|^p\Big)^\frac1{p'}\|G_n-G[\bs u_{\eps\delta}]\|_{L^\infty(Q_T)}
\le C\|G_n-G[\bs u_{\eps\delta}]\|_{L^\infty(Q_T)}.\end{multline}

Going back to \eqref{gn} and using \eqref{qqq} and the estimate above, we get
\begin{multline*}
	\langle\A_\delta\bs u_{\eps\delta}, \bs u_{\eps\delta}-\bs u\rangle\le \int_0^T\langle\partial_t\bs u_n,\bs u_n-\bs u_{\eps\delta}\rangle_p+\langle\A_\delta \bs u_{\eps\delta}, \bs u_n-\bs u\rangle\\
	+C\|G_n-G[\bs u_{\eps\delta}]\|_{L^\infty(Q_T)}
	-\int_{Q_T}\bs f\cdot(\bs u_n-\bs u_{\eps\delta}).
\end{multline*}

So, noticing that $G[\bs u_{\eps\delta}]\underset{\eps\rightarrow0}{\longrightarrow}G[\bs u_{\delta}]$ in $\C\big([0,T];L^\infty(\Omega)\big)$ and recalling \eqref{tende},
\begin{multline*}
\varlimsup_{\eps\rightarrow0}	\langle \A_\delta\bs u_{\eps\delta}, \bs u_{\eps\delta}-\bs u\rangle\le \int_0^T\langle\partial_t\bs u_n,\bs u_n-\bs u_\delta\rangle_p+\int_0^T\langle\bs\chi_\delta+\delta\bs\Lambda_\delta, L(\bs u_n-\bs u)\rangle\\
+\int_0^T\langle\bs{\Upsilon}_\delta,\bs u_n-\bs u\rangle
+C\|G_n-G[\bs u_\delta]\|_{L^\infty(Q_T)}
	-\int_{Q_T}\bs f\cdot(\bs u_n-\bs u_\delta).
\end{multline*}
\pagebreak

{\bf Step 2:} The limit when $\delta\to0$

Because there exists a positive constant $C$, independent of $\eps$ and $\delta$ such that
\begin{align*}
&\|\bs\chi_\delta\|_{\V_p'}\le\varliminf_{\eps\rightarrow0}\|\bs a(L\bs u_{\eps\delta})\|_{\V_p'}\le C\\
&\|\bs\Upsilon_\delta\|_{\V_p'}\le\varliminf_{\eps\rightarrow0}\|\bs b(\bs u_{\eps\delta})\|_{L^{p'}(Q_T)^m}\le C\\
&\|\bs\Lambda_\delta\|_{L^{p'}(Q_T)^\ell}\le\varliminf_{\eps\rightarrow0}\||L\bs u_{\eps\delta}|^{p-2}L\bs u_{\eps\delta}\|_{L^{p'}(Q_T)^\ell}\le C,
\end{align*}
then, for a subsequence,
\begin{align}\label{delta00}
	\nonumber&\bs \chi_\delta\underset{\delta\rightarrow0}{\lraup}\bs\chi\quad\text{ in }\V_p'\\
	\nonumber&\bs \Upsilon_\delta\underset{\delta\rightarrow0}{\lraup}\bs\Upsilon\quad\text{ in }\V_p'\\
	&\delta\bs \Lambda_\delta\underset{\delta\rightarrow0}{\longrightarrow}0\quad\text{ in }L^{p'}(Q_T)^\ell
\end{align}
and, since $G[\bs u_{\delta}]\underset{\delta\rightarrow0}{\longrightarrow}G[\bs u]$ in $L^\infty(Q_T)$, then
\begin{multline*}
\varlimsup_{\delta\rightarrow0}	\varlimsup_{\eps\rightarrow0}	\langle \A_\delta\bs u_{\eps\delta}, \bs u_{\eps\delta}-\bs u\rangle\le \int_0^T\langle\partial_t\bs u_n,\bs u_n-\bs u\rangle_p+\int_0^T\langle\bs\chi, L(\bs u_n-\bs u)\rangle\\
+\int_0^T\langle\bs{\Upsilon},\bs u_n-\bs u\rangle
+C\|G_n-G[\bs u]\|_{L^\infty(Q_T)}
-\int_{Q_T}\bs f\cdot(\bs u_n-\bs u).
\end{multline*}
Letting $n\rightarrow\infty$ in the above inequality, using that $\displaystyle \int_0^T\langle\partial_t\bs u_n,\bs u_n-\bs u\rangle_p\le0$ and   $G_n\underset{n}{\longrightarrow}G[\bs u]$ in $\C\big([0,T];L^\infty(\Omega)\big),$ we conclude that
$$\varlimsup_{\delta\rightarrow0}	\varlimsup_{\eps\rightarrow0}	\langle \A_\delta\bs u_{\eps\delta}, \bs u_{\eps\delta}-\bs u\rangle\le0.$$	

If $\A:\V_p\rightarrow\V_p'$ is defined by
\begin{equation}\label{OpA}
\langle\A\bs v,\bs w\rangle=\int_{Q_T}\bs a(L\bs v)\cdot L\bs w+\int_{Q_T}\bs b(\bs v)\cdot\bs w,
\end{equation}
and so
\begin{equation*}
\langle \A\bs u_{\eps\delta}, \bs u_{\eps\delta}-\bs u\rangle
=\langle \A_\delta\bs u_{\eps\delta}, \bs u_{\eps\delta}-\bs u\rangle
	-\delta\int_{Q_T}|L\bs u_{\eps\delta}|^{p-2}L\bs u_{\eps\delta}\cdot L(\bs u_{\eps\delta}-\bs u),
\end{equation*}
yields
\begin{equation*}
\varlimsup_{\delta\rightarrow0}	\varlimsup_{\eps\rightarrow0}	\langle \A\bs u_{\eps\delta}, \bs u_{\eps\delta}-\bs u\rangle
=\varlimsup_{\delta\rightarrow0}	\varlimsup_{\eps\rightarrow0}\langle \A_\delta\bs u_{\eps\delta}, \bs u_{\eps\delta}-\bs u\rangle\le0.
\end{equation*}
 
 {\bf Step 3:} Conclusion
 
Let $a_\delta=\displaystyle\varlimsup_{\eps\rightarrow0}\langle \A\bs u_{\eps\delta}, \bs u_{\eps\delta}-\bs u\rangle.$ As $\displaystyle\varlimsup_{\delta\rightarrow0}a_\delta\le0$, given $\eta>0$, there exists $\delta_\eta>0$ such that $a_{\delta_\eta}<\frac\eta2$. But, because 
$\displaystyle\varlimsup_{\eps\rightarrow0}	\langle \A\bs u_{\eps\delta_\eta}, \bs u_{\eps\delta_\eta}-\bs u\rangle=a_{\delta_\eta}$, we can find $\eps_\eta=\eps(\delta_n)>0$ for which $\langle \A\bs u_{\eps_\eta\delta_\eta}, \bs u_{\eps_\eta\delta_\eta}-\bs u\rangle<\eta$. So,
\begin{equation*}
\varlimsup_{\eta\to0}\langle \A\bs u_{\eps_\eta\delta_\eta}, \bs u_{\eps_\eta\delta_\eta}-\bs u\rangle\le0.
\end{equation*}

From now on we denote $\bs u_\eta=\bs u_{\eps_\eta\delta_\eta}$ and $k_\eta=k_{\eps_\eta}$.
 
As the operator $\A $ 
is bounded, monotone and hemicontinuous, then it is pseudo-monotone. So, as  $\displaystyle\varlimsup_{\eta\rightarrow0}	\langle\A\bs u_{\eta}, \bs u_{\eta}-\bs u\rangle\le0$ then
\begin{equation}\label{adelta}
	\langle\mathcal A \bs u, \bs u-\bs v\rangle\le\varliminf_{\eta\rightarrow0} \langle\mathcal A \bs u_{\eta}, \bs u_{\eta}-\bs v\rangle\quad\forall\bs v\in \K_{G[\bs u]}.
	\end{equation}

Finally we conclude, going back to \eqref{qquase}, that if $\bs v\in\K_{G[\bs u]}$ then
\begin{multline}\label{finalfinal}
\langle\mathcal A \bs u, \bs u-\bs v\rangle\le\,\,\varliminf_{\eta\rightarrow0}\Big(\tfrac12\int_\Omega|\bs u_0-\bs v(0)|^2+\int_0^T\langle\partial_t\bs v,\bs v-\bs u_{\eta}\rangle_p\\
+\int_{Q_T}k_\eta|L\bs u_{\eta}|^{p-2}L \bs u_{\eta} \cdot L(\bs v-\bs u_{\eta})-\int_{Q_T}\bs f\cdot(\bs v-\bs u_{\eta})\Big).
\end{multline}
But, as $\bs v\in\K_{G[\bs u]}$, as in \eqref{qqq} and \eqref{last}
\begin{multline*}
\int_{Q_T}k_\eta|L\bs u_{\eta}|^{p-2}L \bs u_{\eta} \cdot L(\bs v-\bs u_{\eta})\le
\int_{Q_T}k_\eta|L\bs u_{\eta}|^{p-1}(|L\bs v|- |L\bs u_{\eta}|)\\
\le
\int_{Q_T}k_\eta|L\bs u_{\eta}|^{p-1}(G[\bs u_{\eta}]-|L\bs u_{\eta}|)
+ \int_{Q_T}k_\eta|L\bs u_{\eta}|^{p-1}(G[\bs u]-G[\bs u_{\eta}])\\
 \le C\|G[\bs u]-G[\bs u_{\eta}]\|_{\C([0,T];L^\infty(\Omega))}.
\end{multline*}

Then
\begin{equation*}
\varliminf_{\eta\rightarrow0}\int_{Q_T}k_\eta|L\bs u_{\eta}|^{p-2}L \bs u_{\eta} \cdot L(\bs v-\bs u_{\eta})\le0
\end{equation*}
because $\displaystyle\lim_{\eta\rightarrow0}G[\bs u_{\eta}]=G[\bs u]$ in $\C\big([0,T];L^\infty(\Omega)\big)$ and so 
\begin{equation*}
	\langle\mathcal A \bs u, \bs u-\bs v\rangle\le\tfrac12\int_\Omega|\bs u_0-\bs v(0)|^2+\int_0^T\langle\partial_t\bs v,\bs v-\bs u\rangle
-\int_{Q_T}\bs f\cdot(\bs v-\bs u),
\end{equation*}
concluding the proof, since we already know that $\bs u\in\K_{G[\bs u]} $.
\end{proof}

\begin{proof}[Proof of Theorem \ref{th2.2}]
Let $\bs u_1,\bs u_2\in\K_g$ be two solutions of \eqref{iqv} and denote by $\{\bs w_n\}_n$ and by $\{g_n\}_n$ the regularizing sequences of Lemma~\ref{RegSeq} of $\bs w=\frac{\bs u_1+\bs u_2}{2}\in\K_g$ and $g$, respectively, with $\bs z=\bs u_0$.
Considering  
\begin{equation*}
\eps_n=\|g_n-g\|_{\C([0,T];L^\infty(\Omega))}\quad\text{and}\quad \rho_n=\frac{g_*}{g_*+\eps_n}\underset{n}{\longrightarrow}1,
\end{equation*}
we have $\widehat{\bs w}_n=\rho_n\bs w_n\in\K_g\cap\Y_p$
and it may be chosen as test function in \eqref{iqv} for $\bs u_1$ and $\bs u_2$.
We obtain, by addition, 
\begin{equation}\label{uni}
2\int_0^T\langle\partial_t\widehat{\bs w}_n,\bs w -\widehat{\bs w}_n\rangle_p
+\langle\A\bs u_1,\bs u_1-\widehat{\bs w}_n\rangle
+\langle\A\bs u_2,\bs u_2-\widehat{\bs w}_n\rangle
\le 2\int_{Q_T}\bs f\cdot(\bs w-\widehat{\bs w}_n)
+(\rho_n-1)^2\int_\Omega|\bs u_0|^2. 
\end{equation}

Observing that 
\begin{equation*}
\langle\partial_t\widehat{\bs w}_n,\widehat{\bs w}_n-\bs w \rangle_p
= \rho_n\langle\partial_t\bs w_n,\bs w_n-\bs w \rangle_p
+ \rho_n(\rho_n-1)\langle\partial_t\bs w_n,\bs w_n\rangle_p,
\end{equation*}
integrating in time, since $\bs w_n\in W^{1,p}(0,T;\X_p)\subset \C\big(0,T; L^2(\Omega)^m\big)$, we have
\begin{multline*}
\int_0^T\langle\partial_t\widehat{\bs w}_n,\widehat{\bs w}_n-\bs w \rangle_p
= \rho_n\int_0^T\langle\partial_t\bs w_n,\bs w_n-\bs w \rangle_p\\
+ \tfrac12\rho_n(1-\rho_n)\big(\|\bs u_0\|^2_{L^2(\Omega)^m}-\|\bs w_n(T)\|^2_{L^2(\Omega)^m}\big)
\leq \tfrac12\rho_n(1-\rho_n)\|\bs u_0\|^2_{L^2(\Omega)^m}
\underset{n}{\longrightarrow}0.
\end{multline*}

Therefore, taking the limit in \eqref{uni}, since $\widehat{\bs w}_n\underset{n}{\longrightarrow}\frac{\bs u_1+\bs u_2}{2}$, we obtain
\begin{equation*}
\int_{Q_T}(\bs a(L\bs u_1)-\bs a(L\bs u_2))\cdot(L\bs u_1-L\bs u_2)
+ \int_{Q_T}(\bs b(\bs u_1)-\bs b(\bs u_2))\cdot(\bs u_1-\bs u_2)
=\langle \A\bs u_1-\A\bs u_2,\bs u_1-\bs u_2\rangle
\leq0
\end{equation*}
and the conclusion $\bs u_1=\bs u_2$ follows by the strict monotonicity of $\bs b$ 
or $\bs a$ with the Assumption~\ref{op:a}.
\end{proof}

\section{Solution of the variational inequality}

We study now the variational inequality case as well as the continuous dependence of its solution on the given data. We obtain different stability results whether we consider the case where the operator $\bs a$ is monotone or strongly monotone.

\begin{proof} {\bf of Theorem \ref{IV0}}
We penalise  the variational inequality using the function $k_\eps$ defined in \eqref{kapa_eps}, as we have done in Section~\ref{app-section}. For $\eps\in(0,1)$ and $\delta>0$ let us consider the problem of finding $\bs w_{\varepsilon\delta}\in \V_p\,\cap\, H^1\big(0,T;L^2(\Omega)\big)$ such that 
\begin{equation}\label{ApproxProblem}
\left\{\begin{array}{l}
\displaystyle\int_\Omega\partial_t\bs w_{\eps\delta}(t)\cdot\bs\psi
+\int_\Omega \bs a(t,L\bs w_{\eps\delta}(t))\cdot L\bs\psi+\int_\Omega \bs b(t,\bs w_{\eps\delta}(t))\cdot L\bs\psi\vspace{3mm}\\
\hspace{1cm}\hfill{\quad\quad+\displaystyle\int_\Omega\Big(\delta+ k_\varepsilon\big(|L \bs w_{\eps\delta}(t)|^p -g(t)^p\big)\Big) \left|L\bs w_{\eps\delta}(t)\right|^{p-2}L \bs w_{\eps\delta}(t)\cdot L\bs\psi}\vspace{3mm}\\
\hfill{\quad\quad\displaystyle=\int_\Omega\bs f(t)\cdot\bs\psi\qquad\forall \bs\psi\in\X_p,\quad\text{for a.e. }t\in(0,T)}\vspace{3mm}\\
\bs w_{\eps\delta}(0) =\bs w_0.
\end{array}
\right.
\end{equation}
The proof of existence of solution for this problem is similar to the proof of Proposition~\ref{existence-iv} and can be done with the Galerkin method (see \cite[p.240]{Roubicek2013}, for instance). 
We observe that here we consider the function $k_\eps\big(|L\bs w_{\eps\delta}|^p-g^p\big)$, instead of $k_\eps(|L\bs w_{\eps\delta}|-G[\bs w_{\eps\delta}])$. 

As in the estimates \eqref{est1}, \eqref{est2} and \eqref{3} we obtain, with a constant $C>0$ independent of $\eps$ and $\delta$,
\begin{align}
\label{Est-u_epsilon}
\left\|\bs w_{\varepsilon\delta}\right\|_{L^\infty(0,T;L^2(\Omega)^m)}
& \leq C,\\
\label{Est-Lu_epsilon}
\left\|L \bs w_{\varepsilon\delta}\right\|_{L^p(Q_T)^\ell}
&\le \tfrac{C}{\delta^\frac1p}\\
\label{Est-k_eps}
\|k_\eps(|L\bs w_{\eps\delta}|^p-g^p)\|_{L^1(Q_T)}
&\leq C.
\end{align}

Using Galerkin's approximation, we can argue formally with  $\partial_t\bs w_{\eps\delta}$ as a test function on \eqref{ApproxProblem} and we get
\begin{multline}\label{dtuepdelta}
\int_{\Omega}|\partial_t\bs w_{\varepsilon\delta}(t) |^2
+\int_{\Omega}\bs a(t,L\bs w_{\varepsilon\delta}(t) )\cdot \partial_t L\bs w_{\varepsilon\delta}(t) 
+\int_{\Omega}\bs b(t,\bs w_{\varepsilon\delta}(t) )\cdot \partial_t \bs w_{\varepsilon\delta}(t) \\
+\int_{\Omega}(\delta+k_\varepsilon(|L\bs w_{\eps\delta}|^p-g(t)^p))|L\bs w_{\varepsilon\delta}(t)|^{p-2}L\bs w_{\varepsilon\delta}(t) \cdot \partial_t L\bs w_{\varepsilon\delta}(t) 
=\int_{\Omega}\bs f(t) \cdot\partial_t \bs w_{\varepsilon\delta}(t).
\end{multline}

Set $\displaystyle\phi_\eps(s)=\int_0^sk_\varepsilon(\tau)\,d\tau$ and observe that
\begin{multline*}
k_\varepsilon(|L\bs w_{\eps\delta}(t)|^p-g^p(t))|L\bs w_{\eps\delta}(t)|^{p-2}L\bs w_{\eps\delta}(t)\cdot\partial_t L\bs w_{\eps\delta}(t)\\=
\frac1p k_\varepsilon(|L\bs w_{\eps\delta}(t)|^p-g^p(t))\partial_t\big(|L\bs w_{\eps\delta}(t)|^p-g^p(t)\big)+k_\varepsilon(|L\bs w_{\eps\delta}(t)|^p-g^p(t))g^{p-1}(t)\partial_tg(t)\\
=\frac1p\partial_t \big(\phi_\eps(|L\bs w_{\eps\delta}(t)|^p-g^p(t))+k_\varepsilon(|L\bs w_{\eps\delta}(t)|^p-g^p(t))g^{p-1}(t)\partial_tg(t).
\end{multline*}

Integrating \eqref{dtuepdelta} between $0$ and $T$, we obtain
\begin{multline}\label{dtuepdelta-2}
\int_{Q_T}|\partial_t\bs w_{\eps\delta}|^2+\int_\Omega A(T,L\bs w_{\eps\delta}(T))-\int_\Omega A(0,L\bs w_0)-\int_{Q_T}\big(\partial_t A\big)(L\bs w_{\eps\delta})\\
+\frac{\delta}{p}\int_\Omega| L \bs w_{\eps\delta}(T)|^p
-\frac{\delta}{p}\int_\Omega| L \bs w_0|^p
+\frac1p\int_\Omega \phi_\eps(|L\bs w_{\eps\delta}(T)|^p-g^p(T))\\
-\frac1p\int_\Omega \phi_\eps(|L\bs w_0|^p-g^p(0))
+\int_{Q_T}k_\varepsilon(|L\bs w_{\eps\delta}|^p-g^p)g^{p-1}\partial_tg
=\int_{Q_T}\big(\bs f-\bs b(\bs w_{\eps\delta})\big)\cdot\partial_t \bs w_{\eps\delta},
\end{multline}
since $A$ satisfy \eqref{apot}.
But 
$$\phi_\eps(|L\bs w_{\eps\delta}(T)|^p-g^p(T))\ge0,\quad \phi_\eps(|L\bs w_0|^p-g^p(0))=0\ \text{ because }|L\bs w_0|\le g(0),$$
and, using assumption \eqref{op:a:prop:b:bis}, H\"older and Young inequalities,
\begin{multline*}
\int_{Q_T}\big(\bs f-\bs b(\bs w_{\eps\delta})\big)\cdot\partial_t \bs w_{\eps\delta}\le\int_{Q_T}\big(|\bs f|+b^*|\bs w_{\eps\delta}|\big)|\partial_t\bs w_{\eps\delta}|\\
 \le C\big(\|\bs f\|_{L^2(Q_T)^m}^2+b^*\|\bs w_{\eps\delta}\|_{L^2(Q_T)^m}^2\big)+
\frac12\|\partial_t\bs w_{\eps\delta}\|_{L^2(Q_T)^m}^2
\end{multline*}
and \eqref{op:a:bis}, from \eqref{dtuepdelta-2}, we have
\begin{multline*}
\frac12\|\partial_t\bs w_{\varepsilon\delta}\|_{L^2(\Omega)^m}^2\le C\big(\|\bs f\|_{L^2(Q_T)^m}^2
+b^*\|\bs w_{\eps\delta}\|_{L^2(Q_T)^m}^2\big)\\
+2a^*\|L\bs w_0\|^p_{L^p(\Omega)^\ell}
+A_1|Q_T|+A_2\|L\bs w_{\eps\delta}\|^p_{L^p(Q_T)^\ell}
+\frac\delta{p}\|L\bs w_0\|^p_{L^p(\Omega)^\ell}\\
+\|k_\varepsilon(|L\bs w_{\eps\delta}|^p-g^p)\|_{L^1(Q_T)}\|g\|_{L^p(Q_T)}^{p-1}\|\partial_t g\|_{L^\infty(Q_T)}
\end{multline*}
and from \eqref{Est-u_epsilon}-\eqref{Est-k_eps} we obtain, with a constant $C'>0$ independent of $\eps$ and $\delta$,
\begin{equation} \label{Est-dtu_epsilon}
\|\partial_{t}\bs w_{\varepsilon\delta}\|_{L^{2}(Q_{T})^m}
\le C'\big(1+\|L\bs w_{\eps\delta}\|_{L^p(Q_T)\ell}^p\big).		
\end{equation}

Then, recalling that Assumption~\ref{gelfand} implies, by Aubin-Lions lemma the compactness of $\Y_p\hookrightarrow\HHH$, there exists a subsequence that we still denote by $\{\bs w_{\varepsilon\delta}\}_\eps$ such that, for every $t\in(0,T]$
\begin{align*} 
\bs w_{\eps\delta}&\underset{\eps\rightarrow0}{\longrightarrow}\bs w_\delta\quad\text{in}\quad L^2(Q_T)^m,\\
\nonumber L\bs w_{\eps\delta}&\underset{\eps\rightarrow0}{\lraup}L\bs w_\delta \quad\text{in}\quad L^p(Q_T)^\ell\text{ weak},\\
\nonumber 	 \partial_t\bs w_{\eps\delta}&\underset{\eps\rightarrow0}{\lraup}\partial_t\bs w_\delta\quad\text{in}\quad L^2(Q_T)^m\text{ weak}.
\end{align*}

Recalling Lemma~\ref{monot*} and observing that $k_\eps(|L\bs v|^p-g^p)=0$ if $\bs v\in\K_g$ we have, for any $t\in(0,T]$,
\begin{multline}\label{monotdtw}
\int_{Q_t}k_\eps(|L\bs w_{\eps\delta}|^p-g^p)|L\bs w_{\eps\delta}|^{p-2}L\bs w_{\eps\delta}\cdot L(\bs v-\bs w_{\eps\delta})\\
=\int_{Q_t}\Big(k_\eps(|L\bs w_{\eps\delta}|^p-g^p)|L\bs w_{\eps\delta}|^{p-2}L\bs w_{\eps\delta}
-k_\eps(|L\bs v|^p-g^p)|L\bs v|^{p-2}L\bs v\Big)\cdot L(\bs v-\bs w_{\eps\delta})\\
+\int_{Q_t}k_\eps(|L\bs v|^p-g^p)|L\bs v|^{p-2}L\bs v\cdot L(\bs v-\bs w_{\eps\delta})\le0,\ \forall\bs v\in\K_g
\end{multline}
and using $\bs v -\bs w_{\varepsilon\delta}$ as test function in \eqref{ApproxProblem},  integrating over $(0,t)$, by \eqref{monotdtw} and by the monotonicity of the operators $\bs a$, $\bs b$ and $\bs\xi\mapsto|L\bs\xi|^{p-2}L\bs\xi$, we obtain
\begin{multline*} 
\int_{Q_t}\partial_t\bs w_{\eps\delta}\cdot(\bs v -\bs w_{\eps\delta})
+\int_{Q_t}\bs a(L\bs v)\cdot L(\bs v -\bs w_{\eps\delta})+\int_{Q_t}\bs b(\bs v)\cdot (\bs v -\bs w_{\eps\delta})\\
+\delta\int_{Q_t} |L\bs v|^{p-2}L\bs v\cdot L(\bs v -\bs w_{\eps\delta})
\geq\int_{Q_t}\bs f\cdot(\bs v -\bs w_{\eps\delta})
\end{multline*}
and so, passing to the limit when $\varepsilon$ tends to zero, we get 
\begin{multline} \label{ivfraca}
\int_{Q_t}\partial_t\bs w_{\delta}\cdot(\bs v -\bs w_{\delta})
+\int_{Q_t}\bs a(L\bs v)\cdot L(\bs v -\bs w_{\delta})+\int_{Q_t}\bs b(\bs v)\cdot (\bs v -\bs w_{\delta})\\
+\delta\int_{Q_t} |L\bs v|^{p-2}L\bs v\cdot L(\bs v -\bs w_{\delta})
\geq\int_{Q_t}\bs f\cdot(\bs v -\bs w_{\delta}).
\end{multline}
Arguing as in Lemma~\ref{conv} we also prove that $\bs w_\delta\in\K_g$. 

The next step is to let $\delta\rightarrow0$. 
From \eqref{Est-dtu_epsilon} we have
\begin{equation*}
\|\partial_t \bs w_\delta\|_{L^2(Q_T)^m}
\leq\varliminf_{\eps\to0}\|\partial_t \bs w_{\eps\delta}\|_{L^2(Q_T)^m}
\leq C'\Big(1+\varliminf_{\eps\to0}\int_{Q_T}|L(\bs w_{\eps\delta}|^p\Big).
\end{equation*}

Using the sets defined in \eqref{ABC} we get
\begin{multline*}
\int_{Q_T}|L\bs w_{\eps\delta}|^p=\int_{A_{\eps\delta}}|L\bs w_{\eps\delta}|^p
+
\int_{B_{\eps\delta}}\big(|L\bs w_{\eps\delta}|^p -g^p+g^p\big)
+
\int_{C_{\eps\delta}}|L\bs w_{\eps\delta}|^p\\
\leq 2\|g\|_{L^p(Q_T)}^p
+|Q_T|
+\int_{B_{\eps\delta}}\eps\ k_\eps(|L\bs w_{\eps\delta}|^p -g^p)
+\frac{C_1}{e^\frac1{\eps^2}-1}\leq C,
\end{multline*}
because, for $s\in(0,\frac1\eps)$  $\frac s\eps\leq k_\eps(s)$ and, by \eqref{ajuda},
\begin{equation*}
\int_{C_{\eps\delta}}\big(e^\frac{1}{\eps^2}-1\big)|L\bs w_{\eps\delta}|^p
=\int_{C_{\eps\delta}}k_\eps(|L\bs w_{\eps\delta}|^p -g^p)|L\bs w_{\eps\delta}|^p \leq C_1,
\end{equation*}
so $\{\partial_t\bs w_\delta\}_\delta$ is also uniformly bounded in $L^2(Q_T)^m$.
Since $\bs w_\delta\in\K_g$, we have $|L\bs w_\delta|$ bounded in $L^\infty(Q_T)$ independently of $\delta$.
Then, for a subsequence, we have
\begin{align*}
&\bs w_\delta\underset{\delta\rightarrow0}{\lraup}\bs w\text{ in }H^1\big(0,T;L^2(Q_T)^m\big)\text{ weak},\\
&L\bs w_\delta\underset{\delta\rightarrow0}{\lraup}L\bs w\text{ in }L^\infty(Q_T)^\ell\text{ weak}-*,\\
&\bs w_\delta(t)\underset{\delta\rightarrow0}{\lraup}\bs w(t)\text{ in }L^2(\Omega)^m\text{ weak, for all } 0<t\leq T
\end{align*}
and we can pass to the limit, when $\delta\rightarrow0$, in \eqref{ivfraca}, writing
\begin{multline*}
\int_{Q_t}\partial_t\bs w_\delta\cdot\bs v
+\int_{Q_t}\bs a(L\bs v)\cdot L(\bs v -\bs w_\delta)
+\int_{Q_t}\bs b(\bs v)\cdot (\bs v -\bs w_\delta)\\
+\delta\int_{Q_t}|L\bs v|^{p-2}L\bs v\cdot L(\bs v-\bs w_\delta)
\geq\int_{Q_t}\bs f\cdot(\bs v -\bs w_\delta)
+\frac12\int_\Omega|\bs w_\delta(t)|^2-\frac12\int_\Omega|\bs w_0|^2.
\end{multline*}
Because $\bs w_\delta(t)\lraup\bs w(t)$ in $L^2(\Omega)^m$-weak yields 
$\displaystyle \varliminf_{\delta\to0}\int_\Omega|\bs w_\delta(t)|^2\ge\int_\Omega|\bs w(t)|^2$, for each $0<t\leq T$, we recover, in the limit, that $\bs w$ satisfies 
\begin{equation*} 
\int_{Q_t}\partial_t\bs w\cdot(\bs v -\bs w)
+\int_{Q_t}\bs a(L\bs v)\cdot L(\bs v -\bs w)+\int_{Q_t}\bs b(\bs v)\cdot (\bs v -\bs w)\geq\int_{Q_t}\bs f\cdot(\bs v -\bs w),\quad\forall\bs v\in\K_g.
\end{equation*}
Finally, as in the proof of Theorem~\ref{IQV0}, $\bs w$ also belongs to $\K_g$, we may apply Minty's lemma and conclude that it solves \eqref{iv}.

The uniqueness of solution is immediate since, if $\bs w_1$ and $\bs w_2$ are two solutions of \eqref{iv}, then
\begin{multline*} 
\int_{Q_t}\partial_t(\bs w_1-\bs w_2)\cdot(\bs w_1 -\bs w_2)
+\int_{Q_t}\big(\bs a(L\bs w_1)-\bs a(L\bs w_2)\big)\cdot L(\bs w_1 -\bs w_2)\\
+\int_{Q_t}\big(\bs b(\bs w_1)-\bs b(\bs w_2)\big)\cdot (\bs w_1 -\bs w_2)\le0
\end{multline*}
and, by monotonicity of $\bs a$ and $\bs b$, we get 
$$\int_\Omega|\bs w_1(t)-\bs w_2(t)|^2\le0\quad\text{ for all }t\in(0,T),$$
concluding that $\bs w_1=\bs w_2$.
\end{proof}

Next we prove the stability of the solutions of the variational inequality \eqref{iv} with respect to the given data. The results we obtain depend on the assumptions on $\bs a$, and we are able to give a result even in the very degenerate case $\bs a\equiv\bs0$ and $\bs b\equiv\bs0$.

\begin{proof}{\bf Theorem \ref{stability}}
Considering the threshold functions $g_1$ and $g_2$ satisfying \eqref{assumptions_iv}, let $\K_{g_1(t)}$ and $\K_{g_2(t)}$ be, respectively, the corresponding convex sets defined in \eqref{kapaIV}. 
For $\beta(t)=\|g_i(t)-g_j(t)\|_{L^\infty(\Omega)}$, for $i,j\in\{1,2\}$, $i\neq j$, and given $\bs w_i$ such that $\bs w_i(t)\in\K_{g_i(t)}$ for a.e. $t\in(0,T)$, we define the functions
\begin{equation*}
{\bs w_i}_j(t)=\frac{g_*\bs w_i(t)}{g_* + \beta(t)} \in\K_{g_j(t)}.
\end{equation*}
Choosing $C\geq\max\left\{\frac{1}{g_*}\|\bs w_i\|_{L^2(Q_T)^m},\frac{1}{g_*}\|L\bs w_i\|_{L^p(Q_T)^\ell}\right\}$ we observe that
\begin{equation}\label{uij}
\|\bs w_i(t)-{\bs w_i}_j(t)\|_{L^2(\Omega)^m}\leq C\beta(t)
\quad\text{and}\quad
\|L(\bs w_i(t)-{\bs w_i}_j(t))\|_{L^p(\Omega)^\ell}\leq C\beta(t).
\end{equation}

Considering, for $i,j\in\{1,2\}$, $i\neq j$, the solution $\bs w_i$ of the variational inequality \eqref{iv} associated to the constraint $g_i$, using ${\bs w_j}_i$ as test function, we have, for $t\in(0,T]$,
\begin{equation*}
\int_{Q_t}\partial_{t}\bs w_i\cdot(\bs w_i-{\bs w_j}_i)
+\int_{Q_t}\bs a(L\bs w_i)\cdot L(\bs w_i-{\bs w_j}_i)
+\int_{Q_t}\bs b(\bs w_i)\cdot (\bs w_i-{\bs w_j}_i)
\leq
\int_{Q_t}\bs f_i\cdot(\bs w_i-{\bs w_j}_i)
\end{equation*}
and so
\begin{multline*}
\int_{Q_t}\partial_{t}\bs w_i\cdot(\bs w_i-{\bs w_j})
+\int_{Q_t}\bs a(L\bs w_i)\cdot L(\bs w_i-{\bs w_j})
+\int_{Q_t}\bs b(\bs w_i)\cdot (\bs w_i-{\bs w_j})\\
\leq
\int_{Q_t}\bs f_i\cdot(\bs w_i-{\bs w_j})
+\int_{Q_t}\partial_{t}\bs w_i\cdot({\bs w_j}_i-\bs w_j)
+\int_{Q_t}\bs a(L\bs w_i)\cdot L({\bs w_j}_i-\bs w_j)\\
+\int_{Q_t}\bs b(\bs w_i)\cdot ({\bs w_j}_i-\bs w_j)
+\int_{Q_t}\bs f_i\cdot(\bs w_j-{\bs w_j}_i).
\end{multline*}
Adding the inequalities we obtained in the former expression to $(i,j)=(1,2)$ and $(i,j)=(2,1)$, denoting $\bs w=\bs w_1-\bs w_2$ we get
\begin{equation} \label{ContDep2}
\int_{Q_t}\partial_{t}\bs w\cdot\bs w
+\int_{Q_t}\big(\bs a(L\bs w_1)-\bs a(L\bs w_2)\big)\cdot L\bs w
+\int_{Q_t}\big(\bs b(\bs w_1)-\bs b(\bs w_2)\big)\cdot \bs w
\leq
\int_{Q_t}(\bs f_1-\bs f_2)\cdot\bs w
+ \Theta(t),
\end{equation}
where
\begin{multline*}
\Theta(t)
= \int_{Q_t}\partial_{t}\bs w_1\cdot({\bs w_2}_1-\bs w_2)
+\int_{Q_t}\partial_{t}\bs w_2\cdot({\bs w_1}_2-\bs w_1)
+\int_{Q_t}\bs a(L\bs w_1)\cdot L({\bs w_2}_1-\bs w_2)
\\
+\int_{Q_t}\bs a(L\bs w_2)\cdot L({\bs w_1}_2-\bs w_1)
+\int_{Q_t}\bs b(\bs w_1)\cdot ({\bs w_2}_1-\bs w_2)
+\int_{Q_t}\bs b(\bs w_2)\cdot ({\bs w_1}_2-\bs w_1)
\\
+\int_{Q_t}\bs f_1\cdot(\bs w_2-{\bs w_2}_1)
+\int_{Q_t}\bs f_2\cdot(\bs w_1-{\bs w_1}_2).
\end{multline*}

The estimates \eqref{Est-u_epsilon}, \eqref{Est-Lu_epsilon}, \eqref{Est-dtu_epsilon} and \eqref{uij} allow us to conclude that there exists a positive constant $C$ such that, for any $t\in(0,T)$,
\begin{equation} \label{DepCont3}
\Theta(t)\leq C\int_0^t\|g_1(\tau)-g_2(\tau)\|_{L^\infty(\Omega)}\,d\tau.
\end{equation}

From \eqref{ContDep2} and \eqref{DepCont3}, using the monotonicity of $\bs a$ and $\bs b$, we obtain
$$\frac{d\ }{dt}\|\bs w(t)\|_{L^2(\Omega)^m}^2\le\|\bs w(t)\|_{L^2(\Omega)^m}^2+\|\bs f_1(t)-\bs f_2(t)\|_{L^2(\Omega)^m}^2+2C{\|g_1(t)-g_2(t)\|}_{L^\infty(\Omega)}.$$
Applying the Gronwall inequality we obtain
\begin{multline*}\|\bs w_1(t)-\bs w_2(t)\|^2_{L^2(\Omega)^m}\le e^T\Big(\int_0^T\|\bs f_1(t)-\bs f_2(t)\|^2_{L^2(\Omega)^m}dt\\+\|\bs w_{1_0}-\bs w_{2_0}\|_{L^2(\Omega)^m}
+\int_0^T\|g_1(t)-g_2(t)\|_{L^\infty(\Omega)}\,dt\Big),
\end{multline*}
concluding \eqref{ContDep1}.

Consider now the case where $\bs a$ is strongly monotone. 
When $p\geq2$,  after integration, from \eqref{ContDep2} 
and using also \eqref{DepCont3}, we obtain
\begin{multline*} 
\frac12\|\bs w(t)\|_{L^2(\Omega)^m}^2
+a_*\int_0^t\|L\bs w(\tau)\|_{L^p(\Omega)^\ell}^p  d\tau 
\\
\leq
\int_0^t\|\bs f_1(\tau)-\bs f_2(\tau)\|_{L^{2}(\Omega)^m}\|\bs w(\tau)\|_{L^2(\Omega)^m}d\tau
+\frac12\|\bs w(0)\|_{L^2(\Omega)^m}^2
+C\|g_1-g_2\|_{L^1(0,T;L^\infty(\Omega))}
\end{multline*}
and so we conclude that there exists another positive constant $C$, depending on $T$, such that
\begin{multline}\label{p>2}
\|\bs w_1-\bs w_2\|_{L^\infty(0,T;L^2(\Omega)^ m)}^2
+\|\bs w_1-\bs w_2\|_{\V_p(Q_T)}^p    
\leq
C\big(\|\bs f_1-\bs f_2\|_{L^{2}(Q_T)^m}^{2}
\\+\|\bs w_{1_0}-\bs w_{2_0}\|_{L^2(\Omega)^m}^2
+\|g_1-g_2\|_{L^1(0,T;L^\infty(\Omega))}\big),
\end{multline}
obtaining \eqref{ContDep22} when $p\ge2$.

If $1<p<2$, using the strong monotonicity \eqref{a:monot:forte} of the operator $\bs a$, from \eqref{ContDep2} we get
\begin{equation*} 
\int_{Q_t}\partial_{t}\bs w\cdot\bs w
+a_*\int_{Q_t} |L\bs w|^2\big(|L\bs w_1|+|L\bs w_2|\big)^{p-2}
\leq
\int_{Q_t}(\bs f_1-\bs f_2)\cdot\bs w
+ \Theta(t)
\end{equation*}
and applying the reverse H\"older inequality we obtain
\begin{equation} \label{DepCont5}
\int_{Q_t}\partial_{t}\bs w\cdot\bs w
+a_*\Big(\int_{Q_t} |L\bs w|^p\Big)^\frac2p 
\Big(\int_{Q_t}\big(|L\bs w_1|+|L\bs w_2|\big)^p\Big)^{\frac{p-2}{2}}
\leq
\int_{Q_t}(\bs f_1-\bs f_2)\cdot\bs w
+ \Theta(t).
\end{equation}
Since $p<2$ and $|L\bs w_i|\leq  g^*$ a.e.\ in $Q_T$, $i=1,2$, we have 
\begin{equation*}
\Big(\int_{Q_t}\big(|L\bs w_1|+|L\bs w_2|\big)^p\Big)^\frac{p-2}{2}
\geq C,
\end{equation*}
where $C$ is a  positive constant depending only on $\|\bs f_i\|_{L^{p'}(Q_T)^m}$ and $\|\bs w_{i_0}\|_{L^2(\Omega)^m}$.
From \eqref{DepCont5} we obtain
\begin{equation*}
\frac12\|\bs w(t)\|^2_{L^2(\Omega)^m}
+ C a_*\Big(\int_{Q_t} |L\bs w|^p\Big)^\frac2p 
\leq
\int_0^t\|\bs f_1-\bs f_2\|_{L^{2}(\Omega)^m}\|\bs w\|_{L^{2}(\Omega)^m}
+\frac12\|\bs w(0)\|_{L^2(\Omega)^m}^2
+\int_0^t \Theta
\end{equation*}
and so, using  H\"older and Young inequalities, we get the following inequality
\begin{equation}\label{p<2}
\|\bs w\|^2_{L^\infty(0,T;L^2(\Omega)^m)}
+\|\bs w\|^2_{\V^p}
\leq
C\big(\|\bs f_1-\bs f_2\|^2_{L^{2}(Q_T)^m}
+\|\bs w_{1_0}-\bs w_{2_0}\|_{L^2(\Omega)^m}^2
+\|g_1-g_2\|_{L^1(0,T;L^\infty(\Omega))}\big).
\end{equation}
From \eqref{p>2} and \eqref{p<2}, the conclusion follows.
\end{proof}

\begin{proof}[Proof of Theorem \ref{IVweak}] Consider a sequence of solutions $\bs w_n$ given by Theorem \ref{IV0} for a sequence of $g_n\in W^{1,\infty}\big(0,T;L^\infty(\Omega)\big)$ such that 
$$g_n\underset{n}{\longrightarrow}g\quad\text{ in }\C\big([0,T];L^\infty(\Omega)\big).$$

First we show that $\{\bs w_n\}_n$ is relatively compact in $\C\big([0,T];L^2(\Omega)\big)$. For arbitrary $\eps>0$, there exists $\delta>0$, such that
$$\eps_n=\sup_{|\tau-s|<\delta}\|g_n(\tau)-g_n(s)\|_{L^\infty(\Omega)}<\eps$$
for all $n$ sufficiently large and all $\tau,s\in(0,T)$.  Setting $\rho_n=\displaystyle\frac{g_*}{g_*+\eps_n}$, then $\rho_n\bs w_n(s)\in\K_{g_n(\tau)}$ for all $\tau\in(s-\delta,s+\delta)$ and $\rho_n\bs w_n(s)$ can be choosen as test function in \eqref{iv_espaco} for $\bs w_n$ at $t=\tau$, obtaining
\begin{multline}\label{stau}
\int_\Omega\partial_t\bs w_n(\tau)\cdot(\rho_n\bs w_n(s)-\bs w_n(\tau))+\int_\Omega\bs a(\tau,\bs w_n(\tau))\cdot L(\rho_n\bs w_n(s)-\bs w_n(\tau))\\
+\int_\Omega\bs b(\tau,\bs w_n(\tau))\cdot (\rho_n\bs w_n(s)-\bs w_n(\tau))\ge\int_\Omega\bs f(\tau)\cdot(\rho_n\bs w_n(s)-\bs w_n(\tau)).
\end{multline}

Since $0<\rho_n\le1$ and the solutions $\bs w_n\in\K_{g_n}$ are uniformly bounded in $L^\infty\big(0,T;\V_p\cap L^2(\Omega)^m\big)$, from \eqref{stau} for fixed $s$, we can integrate in $\tau$ on $[s,t]$ obtaining
\begin{align*}
\frac12\int_\Omega|\bs w_n(t)-&\bs w_n(s)|^2=\frac12\int_s^t\frac{d\ }{d\tau}\int_\Omega|\bs w_n(\tau)-\bs w_n(s)|^2\\
&\le (\rho_n-1)\int_s^t\int_\Omega\partial_\tau\bs w_n(\tau)\cdot\bs w_n(s)+C_1|t-s|+C\int_s^t\|\bs f(\tau)\|_{L^2(\Omega)^m}\\
&\le (\rho_n-1)\int_\Omega(\bs w_n(t)-\bs w_n(s))\cdot\bs w_n(s)+C'|t-s|^\frac12(T^\frac12+\|\bs f\|_{L^2(Q_T)^m})\\
&\le C''(\eps+|t-s|^\frac12),\quad\text{ for all }|t-s|<\delta
\end{align*}
and all $n$ sufficiently large. Hence $\{\bs w_n\}_n$ is equicontinuous on $[0,T]$  with values in $L^2(\Omega)^m$. Therefore we can take for a subsequence
$$\bs w_n\underset{n}{\longrightarrow}\bs w\quad\text{ in }\C\big([0,T];L^2(\Omega)^m\big)\quad\text{and}\quad
\bs w_n\underset{n}{\lraup}\bs w\quad\text{ in }L^\infty\big(0,T;\bs V_p\big)\text{-weak} *,$$
for some $\bs w$ which is such that $\bs w\in\K_g\cap\V_p\cap\C\big([0,T];L^2(\Omega)^m\big)$ and $\bs w(0)=\bs w_0$. 

We conclude that $\bs w$ is a weak solution to \eqref{iqv} with $\K_g$, by using Minty's Lemma, taking $\bs v_n=\rho_n\bs v\in\K_{g_n},$ for arbitrary $\bs v\in\K_g$ in
\begin{multline*}
\int_0^T\langle\partial_t\bs v_n,\bs v_n-\bs w_n\rangle_p+\int_{Q_T}\bs a(L\bs v_n)\cdot L(\bs v_n-\bs w_n)\\
+\int_{Q_T}\bs b(\bs v_n)\cdot(\bs v_n-\bs w_n)\ge\int_{Q_T}\bs f\cdot(\bs v_n-\bs w_n)-\frac12\int_\Omega|\bs w_0-\bs v_n(0)|^2
\end{multline*}
and we observe that $\bs v_n\underset{n}{\longrightarrow}\bs v$ in $\Y_p$. By uniqueness, all the sequence $\bs w_n\underset{n}{\longrightarrow}\bs w$.
\end{proof}

\def\ocirc#1{\ifmmode\setbox0=\hbox{$#1$}\dimen0=\ht0 \advance\dimen0
by1pt\rlap{\hbox to\wd0{\hss\raise\dimen0
\hbox{\hskip.2em$\scriptscriptstyle\circ$}\hss}}#1\else {\accent"17 #1}\fi}

\end{document}